\documentclass[a4paper, 11pt, headings = small, abstract,numbers=endperiod]{scrartcl}
\usepackage[english]{babel}
\usepackage{amscd,amsmath,amsthm,array,bbm,hhline,mathrsfs, enumerate,dsfont,changes, booktabs,fancybox,calc,textcomp,xcolor,graphicx,bbm,xspace,nicefrac,stmaryrd,url,arcs,listings,nameref}
\usepackage[title]{appendix}
\usepackage[semibold]{libertine}
\usepackage[upint]{libertinust1math}
\usepackage[scr=boondoxupr, scrscaled = 1.0]{mathalpha}

\usepackage[scaled=0.95]{inconsolata}

\DeclareMathAlphabet{\mathsf}{OT1}{\sfdefault}{m}{n}
\SetMathAlphabet{\mathsf}{bold}{OT1}{\sfdefault}{b}{n}
\usepackage[utf8]{inputenc}

\usepackage[shortlabels]{enumitem}
\usepackage{etoolbox}
\usepackage[normalem]{ulem}
\usepackage{mathtools}
\usepackage{geometry}
\usepackage{float}

\usepackage{bm}
\usepackage{tikz}
\usepackage[outdir =./]{epstopdf}
\usepackage[citestyle = numeric-comp,bibstyle = numeric, firstinits=true, natbib = true, backend = bibtex,maxbibnames=99,url=false,doi=false]{biblatex}
\usepackage{authblk}
\numberwithin{equation}{section}

\usepackage{chngcntr}
\counterwithin{figure}{section}

\usepackage{xcolor}
\usepackage[colorlinks=true, allcolors=myteal]{hyperref}
\definecolor{WIMgreen}{RGB}{60 134 132}
\definecolor{red_pers}{RGB}{204 37 41}
\definecolor{UMblue}{RGB}{4 47 86}
\definecolor{myteal}{RGB}{0 123 137}
\definecolor{dartmouthgreen}{rgb}{0.05, 0.5, 0.06}\definecolor{cobalt}{rgb}{0.0, 0.28, 0.67}\definecolor{coolblack}{rgb}{0.0, 0.18, 0.39}
\definecolor{glaucous}{rgb}{0.38, 0.51, 0.71}\definecolor{hooker\'sgreen}{rgb}{0.0, 0.44, 0.0}\definecolor{lemonchiffon}{rgb}{1.0, 0.98, 0.8}\definecolor{oucrimsonred}{rgb}{0.6, 0.0, 0.0}\definecolor{radicalred}{rgb}{1.0, 0.21, 0.37}\definecolor{raspberry}{rgb}{0.89, 0.04, 0.36}\definecolor{royalazure}{rgb}{0.0, 0.22, 0.66}
\definecolor{dex}{RGB}{138 18 34}
\definecolor{darkgreen}{RGB}{0 69 0}
\definecolor{darkblue}{RGB}{0 0 99}
\theoremstyle{plain}
\newtheorem{theorem}{Theorem}[section]
\newtheorem{proposition}[theorem]{Proposition}
\newtheorem{lemma}[theorem]{Lemma}
\newtheorem{corollary}[theorem]{Corollary}

\theoremstyle{definition}
\newtheorem{definition}[theorem]{Definition}
\newtheorem{assump}{Assumption}    
\theoremstyle{remark}
\newtheorem{remark}[theorem]{Remark}

\SetLabelAlign{center}{\hss#1\hss}

\def\E{\mathbb{E}}

\def\N{\mathbb{N}}
\def\P{\mathbb{P}}
\def\N{\mathbb{N}}
\def\Q{\mathbb{Q}}
\def\R{\mathbb{R}}
\definecolor{darkred}{rgb}{0,0.6,0}
\def\X{\mathbf{X}}

\def\cX{\mathcal{X}}

\newcommand{\cE}{\mathcal{E}}
\newcommand{\cF}{\mathcal{F}}
\newcommand{\cG}{\mathcal{G}}

\newcommand{\cL}{\mathcal{L}}

\newcommand{\ep}{\varepsilon}
\newcommand{\cO}{\mathcal{O}}

\newcommand{\Pro}{\mathbb{P}}

\renewcommand{\hat}{\widehat}
\newcommand{\e}{\mathrm{e}}
\renewcommand{\tilde}{\widetilde}%
\renewcommand{\d}{\mathop{}\!\mathrm{d} }

\newcommand{\KL}{\operatorname{KL}}
\newcommand{\1}{\mathbf{1}}
\newcommand{\bb}{\bar{b}}

\restylefloat{table}

\newcommand{\cH}{\mathcal{H}}

\newcommand{\vertiii}[1]{{\left\vert\kern-0.25ex\left\vert\kern-0.25ex\left\vert #1
		\right\vert\kern-0.25ex\right\vert\kern-0.25ex\right\vert}}
\AtBeginDocument{\providecommand*\colonequiv{\vcentcolon\mspace{-1.2mu}\equiv}}

\setkomafont{sectioning}{\normalcolor\bfseries}
\setkomafont{descriptionlabel}{\normalcolor\bfseries}
\setkomafont{section}{\large}
\setkomafont{subsection}{\normalsize}
\setkomafont{subsubsection}{\normalsize}
\setkomafont{paragraph}{\normalsize}
\setkomafont{subparagraph}{\normalsize}
\setkomafont{author}{\large}
\setkomafont{title}{\large\bfseries}
\setkomafont{date}{\normalsize}

\let\originalleft\left
\let\originalright\right
\renewcommand{\left}{\mathopen{}\mathclose\bgroup\originalleft}
\renewcommand{\right}{\aftergroup\egroup\originalright}

\newcommand*\samethanks[1][\value{footnote}]{\footnotemark[#1]}

\allowdisplaybreaks
\bibliography{impulse}
\makeatother
\title{\fontsize{16}{19} \selectfont 		Data-driven optimal stopping: A pure exploration analysis
}
\allowdisplaybreaks
\author{Sören Christensen\thanks{Kiel University, Department of Mathematics, Heinrich-Hecht-Platz 6, 24118 Kiel, Germany.
\newline Email: \href{mailto:christensen@math.uni-kiel.de}{christensen@math.uni-kiel.de}}\qquad Niklas Dexheimer\thanks{Aarhus University, Department of Mathematics, Ny Munkegade 118, 8000 Aarhus C, Denmark. \newline Email: \href{mailto:dexheimer@math.au.dk}{dexheimer@math.au.dk}/\href{mailto:strauch@math.au.dk}{strauch@math.au.dk}\newline }\qquad Claudia Strauch\samethanks}

\begin{document}

\maketitle
	
\begin{abstract}
The standard theory of optimal stopping is based on the idealised assumption that the underlying process is essentially known. 
In this paper, we drop this restriction and study data-driven optimal stopping for a general diffusion process, focusing on investigating the statistical performance of the proposed estimator of the optimal stopping barrier.
More specifically, we derive non-asymptotic upper bounds on the simple regret, along with uniform and non-asymptotic PAC bounds.
Minimax optimality is verified by completing the upper bound results with matching lower bounds on the simple regret.
All results are shown both under general conditions on the payoff functions and under more refined assumptions that mimic the margin condition used in binary classification, leading to an improved rate of convergence. 
Additionally, we investigate how our results on the simple regret transfer to the cumulative regret for a specific exploration-exploitation strategy, both with respect to lower bounds and upper bounds.
\end{abstract}

\noindent
\small{\textit{\href{https://mathscinet.ams.org/mathscinet/msc/msc2020.html}{2020 MSC:}} Primary 93E35; secondary 60J60, 62M05\\
\textit{key words:} Stochastic control, reinforcement learning, simple regret, efficient exploration.}
\normalsize
	
\section{Introduction}\label{sec: intro}
In many applications where, for example, physical constraints are in place, the concept of \emph{model-based} reinforcement learning (RL) is reasonable. 
In contrast to the usual model-free setup of RL, this approach is based on the assumption that there exists an underlying model, whose system parameters are, however, unknown to the agent and have to be learned through interaction with the environment and feedback. This is done by choosing strategies based on the current parameter estimate and trying to minimise the regret, i.e., the gap between the expected reward of the best strategy and the actual reward achieved.
The fact that the model framework can be included in the decision-making process makes this method more efficient and robust than model-free RL, and has the great advantage of allowing an in-depth theoretical analysis based on stochastic control theory (see, e.g., \cite{wangetal20,szpr21,yildiz2021continuous,jia22}). 
In particular, this approach delivers interpretable strategies that are better suited to high-stakes decision making than many ``black box'' procedures. 

In this paper, we show how sophisticated statistical methods can be integrated into model-based RL approaches in such a way that the optimal decision can be learned from the data on a theoretically validated basis.
More precisely, we focus on the study of \emph{optimal stopping problems}, and we provide both a data-driven approach and a proof of the optimality of the established convergence rates for the simple regret. 
		
Optimal stopping problems represent one of the central classes of problems in stochastic control, with applications in a wide variety of areas (cf., e.g., the monographs \cite{Peskir,MR2374974,oksendal2019}). 
The essential question is always how a decision maker can stop a stochastic process in a way that maximises the expected payoff. Standard methods for solving are based on either martingale or Markov methods and lead to variational or free-boundary problems (Stefan problems). 	
Our setting here is based on stopping problems for scalar diffusion processes $\X=(X_t)_{t\geq0}$ of the form 
\begin{equation}\label{eq: sde}
			\d X_t=b(X_t)\d t+\d W_t,\quad X_0=\xi,
\end{equation}
where $b\colon \R\to\R$ and the initial condition $\xi$ is supposed to be independent of the Brownian motion $(W_t)_{t\geq0}$, defined on a probability space $(\Omega, \mathscr F, \P_b)$.
We work under a very general model assumption by supposing that $b$ belongs to some set of sufficiently regular drift coefficients, for fixed constants $\mathbf C\geq 1$  and $A,\gamma>0$ specified as 
\begin{equation*}\label{def:driftclass}
			\Sigma(\mathbf C,A,\gamma)\coloneqq\left\{b \in \operatorname{Lip}_{\operatorname{loc}}(\R)\colon|b(x)| \le\mathbf C(1+|x|),\ \forall|x|>A\colon b(x)\operatorname{sgn}(x)\le -\gamma\right\}.
\end{equation*}
The above conditions in particular imply ergodicity of $\X$ and the existence of a unique invariant measure $\mu_b$ which is absolutely continuous wrt the Lebesgue
measure, and we denote by $\E_b[ \cdot]$ the expectation operator under which \eqref{eq: sde} holds with $\xi\sim\mu_b$ and by $\E_b^0[ \cdot]$ the one with $\xi=0$. %\footnote{\sc{Prüft bitte, ob die Änderungen richtig überall vorgenommen sind!}}	
Given a payoff function $g\colon (0,\infty)\to\R$, the objective is to stop $\X$ with a stopping time $\tau$ at a positive state such that the expected payoff $\E_b^0\left[g(X_\tau)\right]$ is maximised. 
In the case of stopping problems that are repeated in time, it is usually appropriate to maximise the average rate of return per unit of time. This setting appears in a natural way whenever the next stopping problem is started after stopping in the present one; fundamental results and a variety of applications can be found in \cite[Chapter 6]{ferguson2006applications}. More precisely, instead of a stopping time $\tau$, an entire sequence of intervention times $0\le \tau_1<\tau_2<\ldots$ is chosen, where the first stopping problem ends in $\tau_1$ with payoff $g(X_{{\tau_1}-})$, the second starts from the initial state $X_{\tau_1}=0$, and so on. 
By the fixed time $T>0$, this gives the payoff $\sum_{n:\tau_n\leq T}g(X_{\tau_n-})$, and the goal is to maximise the asymptotic growth rate 
\begin{equation*}
\Phi_b(g)\coloneqq\liminf_{T\to\infty}\frac{1}{T}\ \E_b^0\left[\sum_{n:\tau_n\leq T}g(X_{\tau_n-})\right].
\end{equation*}
Not surprisingly, the renewal structure reduces the problem to maximizing the rate of return $\E_b^0\left[g(X_\tau)\right]/\E_b^0\left[\tau\right]$ of a single round. 
Moreover, the general theory provides that the optimal stopping times are of the threshold type, i.e., the optimization must be performed only over stopping times of the form $\tau_x=\inf\left\{t:X_t\geq x\right\}$, leading to rates of return $g(x)/\E_b^0\left[\tau_x\right]$. 
In other words, 
\begin{equation}\label{def:Phi}
			\Phi_b(g)=\max_x \frac{g(x)}{\xi_b(x)}, \quad \text{ where } \xi_b(x)=\E_b^0\left[\tau_x\right].	
\end{equation}
		
\paragraph*{The pure exploration optimal stopping problem}
All the preceding considerations obviously assume that the dynamics of the underlying process $\X$ are known, which is one of the standard assumptions in optimal stopping theory.
Our previous work has shown that the basic idea of combining methods from nonparametric statistics and stochastic control theory provides fully data-driven procedures for several classes of control problems and stochastic processes (see \cite{chris23,chris23+,chris23++}).
The construction of the strategies was always based on a prior analysis of the estimation problem, a question closely related to \emph{pure exploration for bandits}, which has attracted much interest in the machine learning literature (cf., e.g., \cite{bubetal09,bubetal11,garivier16a,jedra} and the bibliographical remarks in Section 33.5 of \cite{latti20}).
There, the goal is to determine the best strategy without trying to maximize
the \emph{cumulative} observations.

More specifically, in a pure exploration problem, the performance of an estimator is usually evaluated not in terms of payoffs, but in terms of the resources (e.g., time) needed to make a reasonable decision. Examples of such problems arise in settings with a preliminary exploration phase, which in our case corresponds to observing $\X$ until time $T$.
This paper continues the analysis of the simple regret for optimal stopping problems and thus answers the question of how long the underlying process must be observed in order to make a meaningful stopping decision.  
In particular, we make the following contributions:
\begin{enumerate}[topsep=0.5pt, partopsep=0.5pt,itemsep=0cm,parsep=0cm]
\item[$\bullet$] We prove minimax optimal bounds for the simple regret;
\item[$\bullet$] we provide statements about the complexity of pure exploration; and
\item[$\bullet$] we establish a link to the cumulative regret.
\end{enumerate}
This is achieved by greatly deepening and refining the nonparametric analysis of our previous work on data-based stochastic control, introducing established statistical techniques (such as the peeling device) and concepts (such as the margin condition) in the context of model-based RL.

\paragraph{Main findings}	
Our analysis in \cite{chris23} has shown that the information necessary to construct a \emph{data-driven} strategy for solving the stochastic control problem is encoded in the invariant density of $\X$. As a central building block for deriving improved upper bounds on the regret, a precise description of the \emph{payoff functions} turns out to be essential. Specifically, we introduce a condition that is closely related to the \emph{margin condition} (also known as the \emph{Tsybakov noise condition}) used in binary classification. In a slightly modified form, it has also been applied in the context of bandit problems (see \cite{loca18}). To formulate it, we need some notation.

\begin{assump}\label{ann:margin}
Let $\Delta_0\in(0,1)$, $n\in\N$, $\eta,\beta>0$, and let $f$ be a continuous function on $(0,\infty)$ fulfilling $\sup_{x\in(0,\infty)}f(x)<\infty$.
We say that $f$ satisfies Assumption \ref{ann:margin} if there exist $x_1,\ldots,x_n\in(0,\infty)$ such that $f(x_i)=\sup_{x\in(0,\infty)} f(x)<\infty$ for all $i=1,\ldots,n$ and
\begin{equation}\label{def:margin}
\forall 0<\Delta\leq \Delta_0, \quad 
\cX_f(\Delta)\subseteq \bigcup_{i=1}^n \bigg(x_i-\frac12\eta\Delta^\beta,\ x_i+\frac12\eta\Delta^\beta\bigg),
\end{equation}
where
\begin{equation*}\label{def:set}	
\cX_f(\Delta)\coloneq \bigg\{x\in(0,\infty): \sup_{y\in(0,\infty)}f(y)- f(x)\leq \Delta \bigg\}, \quad \Delta>0.
\end{equation*}
\end{assump}

The most important parameter in this condition is $\beta$ appearing in the superset defined in \eqref{def:margin}, since it measures the difficulty of identifying the maximum. 
Indeed, the larger $\beta$ is, the smaller the set $\cX_f(\cdot)$ becomes, which intuitively reduces the number of candidates for the maximiser. 
In the sequel, we will propose an estimator $\hat y_T$ of the maximising argument of \eqref{def:Phi}, whose use as a threshold for the associated stopping problem leads to an expected payoff of $\E_b\left[\tfrac{g}{\xi_b}(\hat y_T)\right]$ instead of the maximum possible $\Phi_b(g)$. 
We derive upper bounds on the associated simple regret, holding uniformly
for $b\in\Sigma(\mathbf C,A,\gamma)$ and over some class $\mathcal G_b=\mathcal G_b(\Delta_0,n,\eta,\beta)$ of sufficiently regular payoff functions, in the sense that the associated payoff per time fulfills Assumption \ref{ann:margin} (see \eqref{def:setG} for the precise definition).
In particular, we will show that, for $\beta\ge1$, there exists a constant $C>0$ such that
\begin{equation*}\label{eq:expdecay}
\sup_{b\in\Sigma(\mathbf C,A,\gamma)}\sup_{g\in\mathcal G_b(\Delta_0,n,\eta,\beta)}
\mathbb E_b\left[\Phi_b(g)-\frac{g}{\xi_b}(\hat y_T)\right]\in\cO\left(\e^{-CT}\right),
\end{equation*}
i.e., in some cases even an exponentially decreasing simple regret is achievable. As our upper bounds are for the expected simple regret, we can then also apply them to the exploration-exploitation algorithm presented in \cite{chris23} and improve the results. 
Moreover, these upper bounds hold uniformly over the class of scalar ergodic diffusions with drift $b\in \Sigma(\mathbf C,A,\gamma)$ solving \eqref{eq: sde} and payoff functions satisfying the margin condition, and we show that uniform results in probability hold for an even larger class of payoff functions. 
More precisely, it will be proven that if $\beta\in(0,1)$, it holds for all payoff functions in a set $\mathcal G_{\bar g}(\kappa,T)$ describing some vicinity of $\bar g$ and for any given $\ep>0$, there exists some constant $c_\ep>0$ such that
		\[
\sup_{b\in\Sigma(\mathbf{C},A,\gamma)}\sup_{\bar{g}\in\cG_b(\Delta_0,n,\eta,\beta)}\sup_{g\in\cG_{\bar{g}}(\kappa,T)}\P_b\left(\Phi_b(g)-\frac{g}{\xi_{b}}(\hat{y}_T)>c_\ep T^{-1/(2-2\beta)}\right)\leq \ep.
		\]
The identified convergence rate $T^{-1/(2-2\beta)}$ improves on the rate of $T^{-1/2}$ established in \cite{chris23}, but it is not obvious whether it is optimal.  
We therefore complete our analysis with an examination of lower bounds, and we develop an approach inspired by the proof of minimax lower bounds for bandit problems with finitely many arms (see, e.g., \cite{latti20}).
Main steps are the construction of two suitable hypotheses and the control of the associated Kullback--Leibler divergence. Adopting this approach to our case leads to the task of finding drift functions $b,\bb\in\Sigma(\mathbf C,A,\gamma)$ and a suitable payoff function $g$ such that the maximisers of $g/\xi_b$ and $g/\xi_{\bb}$ are separated in such a way that, for suitable $\delta,\ep>0$,  
\begin{equation*}
\left\{ \Phi_{\bb}(g)-\frac{g(y)}{\xi_{\bb}(y)}\leq\delta \right\}\subseteq\left\{\Phi_b(g)-\frac{g(y)}{\xi_{b}(y)}>\ep \right\}.
\end{equation*}
While it is challenging to specify drift functions $b,\bb$ such that the associated functions $\xi_b,\xi_{\bb}$ (defined via \eqref{def:Phi}) have different shapes, we will actually verify in Theorem \ref{thm:lowmarg} that, for large enough $T$, there exists a constant $c_L>0$ such that
\[
\inf_{\tilde{y}_T}\sup_{b\in\Sigma(\mathbf{C},A,\gamma)}\sup_{\bar{g}\in  \cG_b(\Delta_0,n,\eta,\beta)}\sup_{g\in\cG_{\bar{g}}(\kappa,T)} \Pro_{b}\left(\Phi_b(g)-\frac{g(\tilde{y}_T)}{\xi_{b}(\tilde{y}_T)}> c_L T^{-1/(2-2\beta)}\right)\ge c_L,
\]
where the infimum extends over all estimators $\tilde{y}_T$. 
Finally, we verify that the rate of $T^{-1/2}$ obtained in \cite{chris23} for the general case is also minimax optimal over the set of all suitable payoff functions, cf.~Theorem \ref{thm: lower bound gen}.
		
\paragraph{Outline of the paper}
The paper is structured as follows.
In Section \ref{sec: upper bound1}, we introduce our estimator and provide a statistical analysis of the underlying invariant density and distribution function estimators. Section \ref{sec:upsimp} contains the rigorous introduction of the margin condition and the derivation of uniform non-asymptotic upper bounds for any $p$-th moment of simple regret as well as corresponding uniform PAC bounds. 
Section \ref{sec: lower bound} gives minimax lower bounds for the simple regret under the margin condition and also in a general framework. 
Implications of the results of the previous sections for the cumulative regret, both in terms of minimax lower bounds and upper bounds for a given exploration-exploitation type strategy, are presented in Section \ref{sec: cum}. The paper concludes with some closing remarks and a brief simulation study in Section \ref{sec: fin}.

\section{Upper bounds on the simple regret}\label{sec: upper bound}
The aim of this section is to provide the basic ingredients of a data-driven method for maximising the asymptotic growth rate described in \eqref{def:Phi}. 
Throughout this section, we assume that we are given a continuous record of observations $(X_s)_{0\le s\le T}$, $T>0$, of the uncontrolled diffusion process $\X$ with drift $b\in\Sigma(\mathbf C,A,\gamma)$ solving \eqref{eq: sde}.
Following the approach developed in \cite{chris23}, we first propose an estimator $\hat y_T$ and bound the corresponding expected simple regret, given by
\begin{equation}\label{def:simpreg}
\mathbb E_b\left[\Phi_b(g)-\frac{g}{\xi_b}(\hat y_T)\right],
\end{equation}
for a payoff function $g\in \cG(y_1,\zeta,M,\mathbf{C},A,\gamma)$, where
\begin{equation*}\label{eq: def g general}
\cG(y_1,\zeta,M,\mathbf{C},A,\gamma)\coloneqq\left\{g\in C((0,\infty)):\begin{cases}
				& g(0+)<0, y_1=\inf\{y>0:g(y)>0\},\\
				&\forall b\in \Sigma(\mathbf{C},A,\gamma): \frac{g(y)}{\xi_b(y)}\leq \sup_{z \in(0,\zeta]}\frac{g(z)}{\xi_b(z)} \,\forall y>0
				\\ &\sup_{z \in[y_1,\zeta]}\vert g(z)\vert\leq M
			\end{cases} \right\},
\end{equation*} 
with $0<y_1<\zeta<\infty$ and $0<M<\infty$. 
The set $\cG(\cdot)$ encodes all requirements on a payoff function such that \eqref{def:Phi} holds for all $b\in\Sigma(\mathbf{C},A,\gamma)$ (see Proposition 1 in \cite{chris23}). 
Intuitively, the parameters $y_1$ and $\zeta$ represent a priori known lower and upper bounds, respectively, for the true maximiser.
The construction of the estimator $\hat y_T$ now relies on a simple plug-in idea, exploiting the explicit form of $\xi_b(\cdot)$ introduced in \eqref{def:Phi} which, by Lemma 1 and equation (3.1) in \cite{chris23}, is given as
\begin{equation}\label{eq:xi}
\xi_b(x)\coloneqq 2\int_0^x\frac{1}{\rho_b(y)}\int_{-\infty}^y\rho_b(z)\d z\d y
=2\int_0^x\frac{F_b(y)}{\rho_b(y)}\d y, \quad x>0,
\end{equation}
 where $\rho_b$ and $F_b$ denote the density and distribution function, respectively, of the invariant measure $\mu_b$.
		
\subsection{Finite-sample analysis for invariant density and distribution function estimation}\label{sec: upper bound1}
Given the explicit form of $\xi_b$ in \eqref{eq:xi}, we have to find estimators for $\rho_b$ and $F_b$. 
In order to obtain uniform results over the class of drift functions $\Sigma(\mathbf{C},A,\gamma)$, we define the following nonparametric estimators.
Given a bounded kernel function $K\colon \R\to\R$ with compact support and arbitrary $x\in\R$, $T>0$, set
$K_T(x)\coloneq T^{1/2}K(xT^{1/2})$, and define
\begin{equation}\label{eq:def est}
\hat{\rho}_T(x)\coloneqq\frac{1}{T}\int_0^T K_T(x-X_s)\d s \quad  \text{ and } \quad		\hat{F}_T(x)\coloneq \frac 1 T \int_0^T \1_{(-\infty,x)}(X_s)\d s.
\end{equation}
A data-based counterpart of \eqref{eq:xi}, needed for estimating $g/\xi_b$ for a payoff function $g\in\cG(\cdot)$, is then obtained by setting		
\begin{equation*}
\hat{\xi}_T(x)\coloneqq 2\int_0^x\frac{\hat{F}_T(y)}{\hat{\rho}_T(y)\lor a}\d y\lor\frac{M_1}{2},\quad x\in[y_1,\zeta],
\end{equation*}
where $a,M_1>0$ are constants such that, for all $b\in\Sigma(\mathbf{C},A,\gamma)$,
\[\forall x\in[0,\zeta], \ \rho_b(x)\geq a\quad\mathrm{ and } \quad \forall x\in [y_1,\zeta], \ \xi_b(x)\geq M_1.\]
The existence of such constants has been discussed in \cite{chris23}.
Finally, choose 
\begin{equation}\label{def:yhat}
\hat y_T\in\operatorname{arg max}_{x\in[y_1,\zeta]}\frac{g}{\hat\xi_T}(x).
\end{equation}
Note that the estimators defined in \eqref{eq:def est} do not require any additional tuning parameters, which makes them straightforward to apply, in sharp contrast to usual nonparametric estimators. Furthermore, the following proposition will show that they are fully adaptive uniformly over the class of drift functions studied, and even achieve a parametric rate of convergence, despite being fully nonparametric. In addition, these results will imply bounds on any $p$-th moment of the simple regret of the estimator $\hat{y}_T$, which will not only allow us to derive uniform PAC bounds (see Corollary \ref{cor: pac gen}), but will also be of utmost importance for proving the upper bounds under the margin condition \ref{ann:margin} (stated in Theorem \ref{thm: upper bound}). 

Finally, the following results are all purely non-asymptotic, and in particular allow the previously mentioned PAC bounds to be stated under finite sample assumptions.
The proof of Proposition \ref{prop: est mom} is based on the results of \cite{aeck21} for invariant density estimation and \cite{kuto01} for the estimation of the distribution function. However, while \cite{kuto01} examines the $L^2$ error, we provide results on all $p$-th moments which are, to the best of our knowledge, novel.

\begin{proposition}\label{prop: est mom}
There exist constants $c_1,c_2,c_3>0$ such that, for any $b\in\Sigma(\mathbf{C},A,\gamma)$, $T>0$ and $p\ge1$, the following assertions hold.
\begin{enumerate}
\item[$\operatorname{(a)}$] 
For any $x\in \R$,
\begin{align*}
\E_b\left[\vert\hat{\rho}_T(x)-\rho_b(x)\vert^p\right]&\le c_1^p T^{-p/2}\left( p^{p/2}+ p^p T^{-p/2} \right),\\
\E_b\left[\vert\hat{F}_T(x)-F_b(x)\vert^p\right]&\le c_2^p T^{-p/2}\left( p^{p/2}+ p^p T^{-p/2} \right).
\end{align*}
\item[$\operatorname{(b)}$] 
For any $g\in \cG(y_1,\zeta,M,\mathbf{C},A,\gamma)$, it holds
\[
\E_b\left[\left(\Phi_b(g)-\frac{g}{\xi_b}(\hat{y}_T)\right)^p\right]\leq c_3^p T^{-p/2}\left( p^{p/2}+ p^p T^{-p/2} \right). 
\]
\end{enumerate}
\end{proposition}
\begin{proof}
The proof of part (a) can be found in Appendix \ref{app:A}.
For verifying (b), note that, by definition and abbreviating $\Sigma=\Sigma(\mathbf C,A,\gamma)$,
\begin{align*}
&\sup_{b\in\Sigma}\E_b\left[\left(\Phi_b(g)-\frac{g}{\xi_b}(\hat{y}_T)\right)^p\right]\\
&\quad\le \sup_{b\in\Sigma}\E_b\left[\left(\Phi_b(g)-\frac{g}{\xi_b}(\hat{y}_T)+\frac{g}{\hat{\xi}_T(\hat{y}_T)}-\frac{g}{\hat{\xi}_T}(y^*)\right)^p\right]\\
&\quad\le 2^p\sup_{b\in\Sigma}\E_b\left[\sup_{y\in[y_1,\zeta]}\left\vert\frac{g}{\xi_b}(y)-\frac{g}{\hat{\xi}_T}(y)\right\vert^p\right]\\
&\quad\leq 4^p M^p M_1^{-2p}\sup_{b\in\Sigma}\E_b\left[\sup_{y\in[y_1,\zeta]}\left\vert\xi_b(y)-\hat{\xi}_T(y)\right\vert^p\right]\\
&\quad\leq 8^p M^p M_1^{-2p}\sup_{b\in\Sigma}\E_b\left[\sup_{y\in[y_1,\zeta]}\left\vert\int_0^y \frac{F_b(z)}{\rho_b(z)}-\frac{\hat{F}_T(z)}{\hat{\rho_T}(z)\lor a}\d z\right\vert^p\right]\\
&\quad\le 8^p M^p M_1^{-2p}\zeta^{p-1}\sup_{b\in\Sigma}\sup_{y\in[y_1,\zeta]}\E_b\left[\left\vert \frac{F_b(y)}{\rho_b(y)}-\frac{\hat{F}_T(y)}{\rho_b(y)}+\frac{\hat{F}_T(y)}{\rho_b(y)}-\frac{\hat{F}_T(y)}{\hat{\rho_T}(y)\lor a}\right\vert^p\right]\\
&\quad\le\frac 1 2 16^p M^p M_1^{-2p}\zeta^{p-1}\sup_{b\in\Sigma}\sup_{y\in[y_1,\zeta]}\E_b\left[a^{-p}\left\vert F_b(y)-\hat{F}_T(y)\right\vert^p+a^{-2p}\left\vert\rho_b(y)-\hat{\rho_T}(y)\right\vert^p\right]\\
&\quad\leq\frac 1 2 16^p M^p M_1^{-2p}\zeta^{p-1}\left(a^{-2p}c_1^p +a^{-p}c_2^p\right)T^{-p/2}\left( p^{p/2}+ p^p T^{-p/2}\right),
\end{align*}
where we used the moment bounds from part (a) in the last step. 
This completes the proof of (b).
\end{proof}
		
\subsection{Upper bound on the simple regret under the margin condition}\label{sec:upsimp}
It turns out that the general results on the rate of convergence of the simple regret (recall \eqref{def:simpreg}) obtained in \cite{chris23} can be  improved significantly for certain payoff functions through a more refined analysis.
More concretely, our extension of Proposition 2 in \cite{chris23} contained in Proposition \ref{prop: est mom} implies that, for any $p\geq1$,
\begin{equation}\label{eq:upgen}
\sup_{b\in\Sigma(\mathbf C,A,\gamma)}\sup_{g\in\cG(y_1,\zeta,M,\mathbf{C},A,\gamma)}\E_b\left[\left(\Phi_b(g)-\frac{g}{\xi_b}(\hat y_T)\right)^p\right]^{\frac{1}{p}}\in\cO(T^{-1/2}).
\end{equation}
Now, for $\Delta_0\in(0,1)$, $n\in\N$, $\eta,\beta>0$, define 
\begin{equation}\label{def:setT}
\mathcal{T}(\Delta_0,n,\eta,\beta)\coloneq\left\{f\in C((0,\infty)) \text{ fulfills Assumption }\ref{ann:margin}\right\},
\end{equation}  
i.e., $\mathcal{T}$ contains all continuous functions on $(0,\infty)$ that satisfy the margin condition.
As our interest is focused on maximising the expected payoff per time, which obviously depends on the diffusion process and its drift, it is natural to require conditions also depending on the underlying diffusion process.
Therefore, for any drift function $b\in \Sigma(\mathbf{C},A,\gamma)$, our investigated class of payoff functions is given by
\begin{equation}\label{def:setG}
\cG_b(y_1,\zeta,M,\Delta_0,n,\eta,\beta)\coloneq\left\{g\in\cG(y_1,\zeta,M,\mathbf{C},A,\gamma): \frac{g}{\xi_b}\in \mathcal{T}(\Delta_0,n,\eta,\beta) \right\},  
\end{equation}
i.e., the class of payoff functions satisfying the assumptions of \cite{chris23} together with the margin condition specified in Assumption \ref{ann:margin} for the expected payoff per time.
Some examples of payoff functions fulfilling this condition for varying values of $\beta$ can be found in Figure \ref{fig: payoff}.

\begin{figure}[h!]
\includegraphics[width=\linewidth]{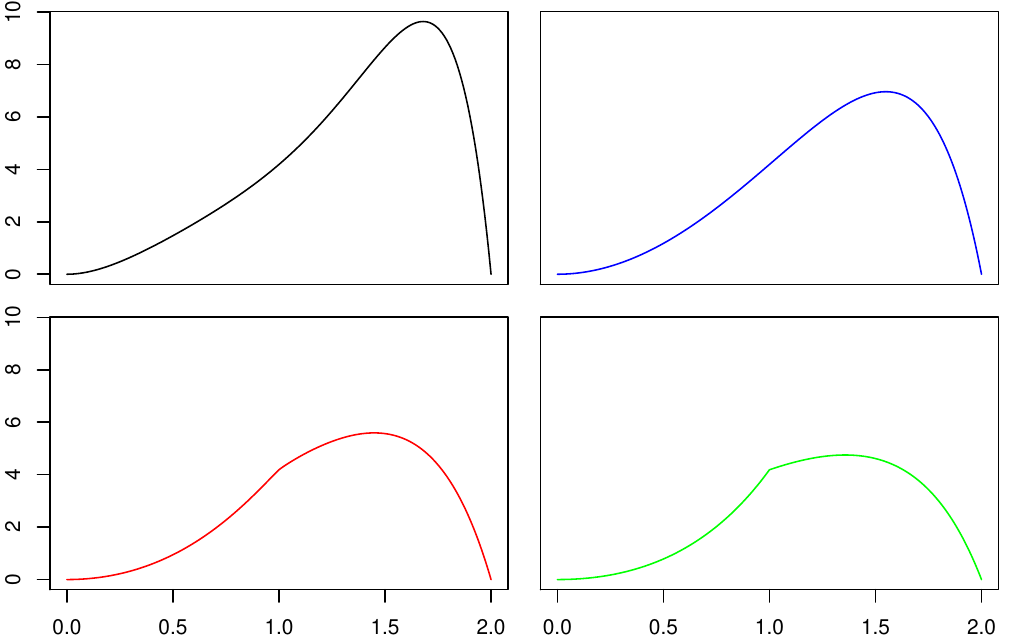}
\caption{Plots of the function $g(x)=(1-\vert1-x\vert^{1/\beta})\xi_b(x),$ for $\beta=0.25, \textcolor{blue}{0.5}, \textcolor{red}{0.75}, \textcolor{green}{1}$, with $b(x)=-x/2$.}
\label{fig: payoff}
\end{figure}

Before stating and proving our upper bounds on the simple regret, we formulate a crucial auxiliary result on moment bounds, restricted to the event $\hat{y}_T\in\mathcal{E}(\Delta),$ where
\[
\mathcal E(\Delta)\coloneqq	\left\{y\in[y_1,\zeta]\colon\Phi_b(g)-\frac{g}{\xi_b}(y)\le \Delta\right\}, \quad \Delta\in(0,1).
\]
Its proof, which is deferred to Appendix \ref{app:A}, uses both the powerful moment bounds of Proposition \ref{prop: est mom} and implications of the margin condition specified in Assumption \ref{ann:margin}.
		
\begin{lemma}\label{lemma: local upper bound}
Let $0<y_1<\zeta<\infty$, $n\in \N$, $M,\eta,\beta>0$, and $\Delta_0\in(0,1)$ be given. 
Then, there exist constants $c_{1,*},c_{2,*}>0$ such that, for any $p\geq 1$, $T>0$, $0<\Delta\leq\Delta_0$, and $\hat y_T$ defined according to \eqref{def:yhat}, we have
\begin{equation*}
\sup_{b\in\Sigma(\mathbf{C},A,\gamma)}\sup_{g\in  \cG_b} \E_b\left[ \left\vert\Phi_b(g)-\frac{g}{\xi_b}(\hat{y}_T) \right\vert^p\1_{ \cE(\Delta)}(\hat{y}_T)  \right]
\leq c_{1,*}c_{2,*}^p(\Delta^p+\Delta^{\beta p})T^{-\frac p2}\left( p^{\frac p2}+ p^p T^{-\frac p2}\right),
\end{equation*}
where $\mathcal G_b=\cG_b(y_1,\zeta,M,\Delta_0,n,\eta,\beta)$ is defined according to \eqref{def:setG}.
\end{lemma}
		
Combining the previous lemma with an application of the peeling device (see \cite{vandegeer}), we are able to show the following results for the simple regret under the margin condition.
		
\begin{theorem}\label{thm: upper bound}
Grant the assumptions of Lemma \ref{lemma: local upper bound}.
Then, there exists a constant $C_1\geq 1$ such that, for any $T>0$, $\beta>0$, $p\geq 1$, and $\mathcal G_b(\beta)\coloneqq \cG_b(y_1,\zeta,M,\Delta_0,n,\eta,\beta)$ defined according to \eqref{def:setG},
\begin{equation}\begin{split}\label{eq: upp 1}
&\sup_{b\in\Sigma(\mathbf{C},A,\gamma)}\sup_{g\in  \cG_b(\beta)}\E_b\left[\left(\Phi_b(g)-\frac{g}{\xi_{b}}(\hat{y}_T)\right)^p\right]\\
&\quad\leq \inf_{0< \alpha< \beta\land 1}C_1^{\frac{p}{1-\alpha}}T^{-\frac{p}{2-2\alpha}}\left(\left(\frac{p}{1-\alpha}\right)^{\frac{p}{2-2\alpha}}+\left(\frac{p}{1-\alpha}\right)^{\frac{p}{1-\alpha}}T^{-\frac{p}{2-2\alpha}}\right).
\end{split}\end{equation}
In particular, for $0<\beta<1$, 
\begin{equation}\label{eq: upp 2} \sup_{b\in\Sigma(\mathbf{C},A,\gamma)}\sup_{g\in \cG_b(\beta)}\E_b\left[\Phi_b(g)-\frac{g}{\xi_{b}}(\hat{y}_T)\right]\in\cO\left(T^{-\frac{1}{2-2\beta}}\right).
\end{equation}
For $\beta\geq 1$, there exists a constant $C_2>0$ such that
\begin{equation}\label{eq: upp 3}
\sup_{b\in\Sigma(\mathbf{C},A,\gamma)}\sup_{g\in \cG_b(\beta)}\E_b\left[\Phi_b(g)-\frac{g}{\xi_{b}}(\hat{y}_T)\right]\in\cO\left(\e^{-C_2T}\right).
\end{equation}
\end{theorem}
\begin{proof}
We first consider the case $\beta\in(0,1)$. 
Let $u\geq 1$ be given. 
For $r\geq 1$, $T>0$, set 
   \[\theta(r,T)\coloneq T^{-\frac{r}{2-2\beta}} \left(\left(\frac{r}{1-\beta}\right)^{r/(2-2\beta)}+\left(\frac{r}{1-\beta}\right)^{r/(1-\beta)}T^{-r/(2-2\beta)}\right) .\]
Then, for preparing the peeling argument, define the sets
\[
B_{j,T}\coloneq \left\{y\in[y_1,\zeta]\colon 2^{j-1}<\theta(u,T)^{-1} \Big(\Phi_b(g)-\frac{g}{\xi_b}(y)\Big)^u \le 2^{j}\right\},\quad j\in\N,\ T>0.
\]
Now,
\begin{equation}\label{ineq:decomp}
\E_b\left[\left(\Phi_b(g)-\frac{g}{\xi_{b}}(\hat{y}_T)\right)^u\right] \le \theta(u,T)+ \operatorname{(I)}+\operatorname{(II)},
\end{equation}
with
\begin{align*}\nonumber
\operatorname{(I)}&\coloneqq 
\sum_{j\in \N: 2^{j}\le\Delta^u_0\theta(u,T)^{-1}}\E_b\left[\Big(\Phi_b(g)-\frac{g}{\xi_{b}}(\hat{y}_T)\Big)^u\ \1_{\hat{y}_T\in B_{j,T}}\right],\\
\operatorname{(II)}&\coloneqq
\E_b\left[\Big(\Phi_b(g)-\frac{g}{\xi_b}(\hat{y}_T)\Big)^u\ \1_{\Phi_b(g)-(g/\xi_b)(\hat{y}_T)\geq \Delta_0 }\right].
\end{align*}
By Lemma \ref{lemma: local upper bound}, it holds for $p=(1+\beta)/(1-\beta)$,
\begin{align}\nonumber
\operatorname{(I)}&\le
\sum_{j\in \N: 2^j\leq\Delta^u_0 \theta(u,T)^{-1}}\theta(u,T)^{-p}2^{-p(j-1)}\E_b\left[(\Phi_b(g)-\frac{g}{\xi_{b}}(\hat{y}_T))^{u(p+1)}\1_{(\Phi_b(g)-(g/\xi_b)(\hat{y}_T))\leq T^{-\frac{1}{2-2\beta}} 2^{j/u}  }\right]\\
&\le \nonumber 2c_{1,*}c_{2,*}^{u(p+1)}
\sum_{j\in \N: 2^j\leq\Delta^u_0 \theta(u,T)^{-1}}\Big(\theta(u,T)^{-p}2^{-p(j-1)}(\theta(u,T) 2^{j})^{\beta (p+1)}T^{-u(p+1)/2}\\\nonumber&\quad\times\left( (u(p+1))^{u(p+1)/2}+ (u(p+1))^{u(p+1)} T^{-u(p+1)/2}\right)\Big)
\\
&=\nonumber 2^{p+1}c_{1,*}c_{2,*}^{u(p+1)} \theta(u,T)^{\beta(p+1)-p}\theta(2u,T)
\sum_{j\in \N: 2^j\leq\Delta^u_0 \theta(u,T)^{-1}}2^{j(\beta(p+1)-p)}
\\
&\leq \nonumber 2^{p+1+2u/(1-\beta)}c_{1,*}c_{2,*}^{u(p+1)} \theta(u,T)^{\beta(p+1)-p+2}\sum_{j\in \N: 2^j\leq\Delta^u_0 \theta(u,T)^{-1}}2^{-j}
\\&\leq 2^{5u/(1-\beta)}c_{1,*}c_{2,*}^{2u/(1-\beta)} \theta(u,T).\label{bound:I}
\end{align}
Furthermore, by Proposition \ref{prop: est mom}, we obtain for $q=1/(1-\beta)$
\begin{align}
\operatorname{(II)}&\le \Delta_0^{-u(q-1)}\nonumber
\E_b\left[\left(\Phi_b(g)-\frac{g}{\xi_b}(\hat{y}_T)\right)^{uq}\right]\\
&\le\nonumber \Delta_0^{-u(q-1)}c_3^{uq}T^{-uq/2}\left((uq)^{uq/2}+(uq)^{uq}T^{-uq/2}\right)
\\
&= \Delta_0^{-u\beta/(1-\beta)}c_3^{u/(1-\beta)}\theta(u,T).\label{bound:II}
\end{align}
Plugging the associated bounds \eqref{bound:I} and \eqref{bound:II} into \eqref{ineq:decomp} already gives \eqref{eq: upp 1} and \eqref{eq: upp 2}, since $0<\beta_1\leq\beta_2$ implies
\[\cG(\beta_2)=\cG_b(y_1,\zeta,M,\Delta_0,n,\eta,\beta_2)\subseteq\cG_b(y_1,\zeta,M,\Delta_0,n,\eta,\beta_1)=\cG(\beta_1). \]
For proving \eqref{eq: upp 3}, let $\beta_T= 1-c_{\exp}/T,$ for some constant $c_{\exp}>0$ to be specified later. 
Then, for large enough $T$, it holds $\beta_T\in(0,1)$. Hence, \eqref{eq: upp 1} gives
\begin{align*}
&\sup_{b\in\Sigma(\mathbf{C},A,\gamma)}\sup_{g\in  \cG_b(\beta)}\E_b\left[\left(\Phi_b(g)-\frac{g}{\xi_{b}}(\hat{y}_T)\right)^p\right]\\
&\quad\leq C_1^{\frac{p}{1-\beta_T}}T^{-\frac{p}{2-2\beta_T}}\left(\left(\frac{p}{1-\beta_T}\right)^{\frac{p}{2-2\beta_T}}+\left(\frac{p}{1-\beta_T}\right)^{\frac{p}{1-\beta_T}}T^{-\frac{p}{2-2\beta_T}}\right)
\\
&\quad=C_1^{\frac{pT}{c_{\exp}}}T^{-\frac{pT}{2c_{\exp}}}\left(\left(\frac{pT}{c_{\exp}}\right)^{\frac{pT}{2c_{\exp}}}+\left(\frac{pT}{c_{\exp}}\right)^{\frac{pT}{c_{\exp}}}T^{-\frac{pT}{2c_{\exp}}}\right)
\\
&\quad=\exp\left(\log\left(\frac{C_1^2p}{c_{\exp}}\right)\frac{Tp}{2c_{\exp}}\right)+\exp\left(\log\left(\frac{C_1p}{c_{\exp}}\right)\frac{Tp}{c_{\exp}}\right).
\end{align*}
Thus, choosing $c_{\exp}=2pC_1^2$ proves the assertion.
\end{proof}

Lastly, as Theorem \ref{thm: upper bound} allows us to bound all moments of the simple regret, we are able to derive non-asymptotic concentration inequalities, more specifically PAC bounds, for the performance of our estimator.

\begin{corollary}\label{cor: pac}
Let everything be given as in Theorem \ref{thm: upper bound}.
Then, for any $T>0$, $u\ge1$, it holds
\begin{equation}\begin{split}\label{eq: pac}
				\sup_{b\in\Sigma(\mathbf{C},A,\gamma)}\sup_{g\in  \cG_b(\beta)}\P_b\left(\Phi_b(g)-\frac{g}{\xi_{b}}(\hat{y}_T)
				\geq \inf_{0< \alpha< \beta\land 1} \Psi_{\alpha,T}(u)\right)\leq \exp(-u),
\end{split}\end{equation}  
where, for $\alpha\in(0,1)$, $u\ge 1$, $T>0$,
\[
\Psi_{\alpha,T}(u)\coloneqq \e C_1^{\frac{1}{1-\alpha}}T^{-\frac{1}{2-2\alpha}}\left(\left(\frac{u}{1-\alpha}\right)^{\frac{1}{2-2\alpha}}+\left(\frac{u}{1-\alpha}\right)^{\frac{1}{1-\alpha}}T^{-\frac{1}{2-2\alpha}}\right). 
\]
In particular, for any $\delta\in(0,\e^{-1}]$, $\ep\in(0,1)$, the uniform PAC bound
\[
\sup_{b\in\Sigma(\mathbf{C},A,\gamma)}\sup_{g\in  \cG_b(\beta)}\P_b\left(\Phi_b(g)-\frac{g}{\xi_{b}}(\hat{y}_T)
\geq\ep\right)\leq \delta
\]
holds, if $\beta\in(0,1)$, for any 
	\begin{equation}\label{eq: PAC1}
			T \ge \frac{4C_1^2\e^{2-2\beta}\log(\delta^{-1})}{(1-\beta)\ep^{2-2\beta}},
	\end{equation}
and, if $\beta\ge1$, for any
	\begin{equation}\label{eq: PAC2}
		T \ge 
		\frac{4C_1^2}{\log(2)}\log(2\delta^{-1})\log(\e\ep^{-1}).
	\end{equation}
	\end{corollary}
	\begin{proof}
Inequality \eqref{eq: pac} directly follows from Theorem \ref{thm: upper bound} together with Markov's inequality. 
Now, for $\beta\in(0,1)$, \eqref{eq: pac} implies that for any $\delta\in(0,\e^{-1}]$, $T>0$,
\begin{equation*}
\sup_{b\in\Sigma(\mathbf{C},A,\gamma)}\sup_{g\in  \cG_b(\beta)}\P_b\left(\Phi_b(g)-\frac{g}{\xi_{b}}(\hat{y}_T)
\geq \Psi_{\beta,T}(\log(\delta^{-1}))\right)\leq \delta,
\end{equation*}
and straightforward computations show that, for $\ep\in(0,1)$, it holds $\Psi_{\beta,T}(\log(\delta^{-1}))\leq \ep$ for any $T$ fulfilling \eqref{eq: PAC1}, concluding the proof of the second assertion.
Lastly, let again $\delta\in(0,\e^{-1}]$ be given. 
Then, arguing as in the proof of \eqref{eq: upp 3} and applying Markov's inequality, one obtains for $T\ge 2\log(\delta^{-1})C_1^2$ that
\[
\sup_{b\in\Sigma(\mathbf{C},A,\gamma)}\sup_{g\in  \cG_b(\beta)}\P_b\left(\Phi_b(g)-\frac{g}{\xi_{b}}(\hat{y}_T)		\geq  \exp\left(1+\frac{\log(2)}{\log(\delta^{-1})}\left(1-\frac{T}{4C_1^2}\right)\right)\right)\leq \delta. 
\]
Noting that, for any $\ep\in(0,1)$,
\[T\geq \frac{4C_1^2}{\log(2)}\log(2\delta^{-1})\log(\e\ep^{-1})\quad\implies \quad \exp\left(1+\frac{\log(2)}{\log(\delta^{-1})}\left(1-\frac{T}{4C_1^2}\right)\right)\leq \ep \]
concludes the proof.
\end{proof}

Let us briefly comment on how the above results relate to the literature on bandit problems.

\begin{remark}
\begin{enumerate}[(a)]
\item 
The uniform PAC bounds stated in Corollary \ref{cor: pac} can again be embedded in the context of pure exploration in bandit problems. In this setting, one is often interested in studying triplets $(\pi,\tau,\psi)$ consisting of a strategy $\pi$, a stopping time $\tau$, and a recommended action $\psi$, i.e., the agent follows the strategy $\pi$ until the stopping time $\tau$, after which they recommend the arm $\psi$ as the best arm. 
 A triplet of this form is then called \emph{sound} at level $\delta>0$ if the probability of $\tau$ being finite and the arm $\psi$ being suboptimal is less than $\delta$ (for a rigorous definition, see Definition 33.4 in \cite{latti20}). This definition can be analogously interpreted as requiring that the probability of $\tau$ being finite and the simple regret being strictly positive after following the strategy for $\tau$ turns is less than $\delta$. Since there is no hope of identifying an optimal barrier exactly in our problem, due to its non-discrete character, a natural modification of this definition is to require that the simple regret is greater than $\ep>0$ only with probability less than $\delta>0$ after following some estimation procedure up to a stopping time $\tau$. 
 Corollary \ref{cor: pac} then implies that, following the proposed estimation procedure up to the time $T$ specified in \eqref{eq: PAC1}, respectively \eqref{eq: PAC2}, is indeed an $\ep$-sound strategy at level $\delta$. 
 \item
 The stated PAC bounds (respectively bounds on the sample complexity) should be compared with the lower bound on the sample complexity obtained in the pure exploration setting for Bernoulli bandits in \cite{mann04}, in particular regarding its dependence on $\ep$. We see that, for $\beta\in(0,1)$, our upper bound improves the lower bound in \cite{mann04} by a factor of $\ep^{2\beta}$. For $\beta\ge 1$, even an only logarithmic increase in $\ep^{-1}$ is possible instead of $\ep^{-2}$.
 \end{enumerate}
 \end{remark}
 
 Since our results on the simple regret in the general case investigated in \cite{chris23} also hold for all moments (see Proposition \ref{prop: est mom}), we are also able to derive new uniform PAC bounds, or bounds on the sample complexity, which interestingly match the lower bounds presented in \cite{mann04}.
 
 \begin{corollary}\label{cor: pac gen}
    For any $T>0$, $u\ge 1$, it holds
				\[\sup_{b\in\Sigma(\mathbf{C},A,\gamma)}\sup_{g\in \cG(y_1,\zeta,M,\mathbf{C},A,\gamma)}\P_b\left(\Phi_b(g)-\frac{g}{\xi_b}(\hat{y}_T)\geq \e c_3 T^{-1/2}\left( u^{1/2}+ u T^{-1/2} \right)\right)\leq \e^{-u}, \]
    where $c_3>0$ is the same constant as in Proposition \ref{prop: est mom}. In particular, for any $\delta\in(0,\e^{-1}]$, $\ep\in(0,1)$, the uniform PAC bound 
 \[\sup_{b\in\Sigma(\mathbf{C},A,\gamma)}\sup_{g\in \cG(y_1,\zeta,M,\mathbf{C},A,\gamma)}\P_b\left(\Phi_b(g)-\frac{g}{\xi_b}(\hat{y}_T)\geq \ep\right)\leq \delta \]
    holds for any 
    \[T\geq \frac{4\e^2c_3^2}{\ep^2}\log(\delta^{-1}) .\]
 \end{corollary}
 We omit the proof as it is completely analogous to the proof of Corollary \ref{cor: pac}.

\section{Lower bounds for the simple regret}\label{sec: lower bound}
In order to show that the rates of convergence for the simple regret established under the margin assumption (see Section \ref{sec:upsimp}) and in the general case (cf.~Section 4 in \cite{chris23}, resp.\ Proposition \ref{prop: est mom}), respectively, are optimal, we provide minimax lower bounds over the investigated classes of drift and payoff functions. 
As already noted in Section \ref{sec: intro}, our analysis is inspired by the literature on minimax lower bounds for bandit problems with finitely many arms. The basic idea there is to construct two bandit problems such that the Kullback--Leibler divergence of the distributions of the arms is constant wrt the sample size, while the optimal arm is different in the respective bandit problems. 
Applying the Bretagnolle--Huber inequality, which states that for two probability measures $\P,\mathbb{Q}$ on the same measurable space $(\Omega,\mathscr{F})$, it holds
\begin{equation}\label{bh ineq}
\forall A\in\mathscr{F}\colon\quad\P(A)+\mathbb{Q}(A^\mathsf{c})\geq \frac 1 2 \exp(-\KL(\P,\mathbb{Q}))
\end{equation}
(see, e.g., Theorem 14.2 in \cite{latti20} or Lemma 2.6 in \cite{tsy09}) then provides the lower bound. 
In our setting, we thus need to specify $b,\bb\in\Sigma$ and $g$ in such a way that, for suitable $\delta,\ep>0$,  
\begin{equation}\label{eq: idea lower bound}
\left\{ \Phi_{\bb}(g)-\frac{g(y)}{\xi_{\bb}(y)}\leq\delta \right\}\subseteq\left\{\Phi_b(g)-\frac{g(y)}{\xi_{b}(y)}>\ep \right\}.
\end{equation}
The set on the LHS of \eqref{eq: idea lower bound} can then be interpreted as the optimal arm under the drift $\bb$, and \eqref{eq: idea lower bound} indicates that, when this arm is pulled, the suboptimal arm under the drift $b$ is played.
		
\subsection{Lower bound under the margin condition}
To verify the minimax optimality of the convergence rate under the margin condition established in \eqref{eq:cor}, we need to find two drift functions $b,\bb\in\Sigma(\mathbf C,A,\gamma)$ and a single payoff function $g$ satisfying $g\in\cG_b,\cG_{\bb}$, which is far from being straightforward. 
Indeed, applying a Taylor expansion shows that the margin condition is deeply related to the multiplicity of the root of the derivative of $g/\xi_b$, which is known to be highly unstable even under minor perturbations (see, e.g., \cite{dayton11}). 
However, as our upper bound holds in expectation, a straightforward application of Markov's inequality also yields stochastic boundedness for all payoff functions $g$ in the vicinity of the original function $\bar{g}$, with the same rate as in Theorem \ref{thm: upper bound}. 
More precisely, defining for each $b\in\Sigma(\textbf{C},A,\gamma)$, $\bar{g}\in\cG_b(y_1,\zeta,M,\textbf{C},\Delta_0,n,\eta,\beta)$ and $\kappa,T>0$ the class of payoff functions
\begin{equation}\begin{split}\label{def:setGbar}
\cG_{\bar{g}}(\kappa,T)&\coloneq\bigg\{g\in\cG\colon\big\{\Phi_b(g)-\frac{g}{\xi_b}(y)>\kappa c T^{-1/2(1-\beta)}\big\}\\
&\hspace*{6em} \subseteq\big\{\Phi_b(\bar{g})-\frac{\bar{g}}{\xi_b}(y)>c T^{-1/2(1-\beta)}\big\}  \ \forall c>1 \bigg\},
\end{split}\end{equation}
the following result holds.
		
\begin{corollary}\label{corr: op rate}
Let everything be given as in Theorem \ref{thm: upper bound}. 
Then, for any $\ep>0$, there exists a constant $c_\ep>0$ such that, for any $T\geq1$,
\begin{equation}\label{eq:cor} \sup_{b\in\Sigma(\mathbf{C},A,\gamma)}\sup_{\bar{g}\in\cG_b(y_1,\zeta,M,\Delta_0,n,\eta,\beta)}\sup_{g\in\cG_{\bar{g}}(\kappa,T)}\P_b\left(\Phi_b(g)-\frac{g}{\xi_{b}}(\hat{y}_T)>c_\ep T^{-1/(2-2\beta)}\right)\leq \ep.
\end{equation}
\end{corollary}
\begin{proof}
The claim immediately follows from the definition \eqref{def:setGbar} of $\cG_{\bar{g}}$, together with Theorem \ref{thm: upper bound} and Markov's inequality.
\end{proof}
Hence, it is sufficient to show that the payoff function used lies in the vicinity of a suitable payoff function, instead of directly verifying the margin condition.
Applying this technique then leads to the following result, which shows the existence of appropriate drift and payoff functions that satisfy \eqref{eq: idea lower bound} under some minor technical assumptions on the parameters of the classes under consideration.
Given $b\in\Sigma(\mathbf C,A,\gamma)$, we denote by $\Pro^T_b$ the law of the process solving \eqref{eq: sde} induced on the canonical space $\big(C([0,T]; \mathbb R), \mathscr B_{C([0,T];\mathbb R)}\big)$.
Furthermore, define the constant $M_2>0$ such that $\xi_b(x)\leq M_2$ for all $x\in[y_1,\zeta],b\in\Sigma(\mathbf{C},A,\gamma)$.
Recall the definition of the classes $\mathcal G_{b}(\cdot)$ and $\mathcal G_{\bar g}(\kappa,T)$ of payoff functions introduced in \eqref{def:setG} and \eqref{def:setGbar}, respectively.
		
\begin{proposition}\label{prop: existence hyp}
Let $M,\gamma>0$, $\mathbf{C}\geq 1$, $\beta\in(0,1)$, $y^*>M^\beta$, $\Delta_0\in(0,1)$, $n\in\N$, $\eta\ge2$, and assume that the parameters $A,y_1,\zeta,\kappa$ are such that  $A>y^*+M^\beta+\gamma$, $y_1\leq  y^*-M^\beta$, $\zeta\geq y^*+M^\beta$, and
\[\kappa\ge
\left(1+(\beta M/\sqrt{\pi})^{\beta/(1-\beta)}\right) 
\left(2\left(\frac{M}{\sqrt\pi}\lor\frac1\beta\right)+\frac{M}{\sqrt{\pi}}\right).\]
Then, for large enough $T$, there exist $b,\bb\in\Sigma(\mathbf{C},A,\gamma)$ and a payoff function $g$ such that the following statements hold:
\begin{enumerate}
\item[$\operatorname{(a)}$]
There exists some constant $c_{L,1}>0$ such that the associated Kullback--Leibler divergence satisfies
\begin{equation*}
\KL(\Pro^T_{b},\Pro^T_{\bb})\leq c_{L,1};
\end{equation*}
\item[$\operatorname{(b)}$]
there exist constants $c_{L,2},c_{L,3}>0$ such that
\[
\left\{ \Phi_{\bb}(g)-\frac{g(y)}{\xi_{\bb}(y)}\leq c_{L,2} T^{-1/(2-2\beta)}\right\}\subseteq\left\{\Phi_b(g)-\frac{g(y)}{\xi_{b}(y)}>c_{L,3} T^{-1/(2-2\beta)} \right\}; 
\]
\item[$\operatorname{(c)}$]
$g\in \cG_{\bb}(y_1,\zeta,MM_2,\Delta_0,n,\eta,\beta)$ and, for a function $\bar{g}$ such that $\bar{g}\in \cG_b(y_1,\zeta,MM_2,\Delta_0,n,\eta,\beta)$, $g\in\cG_{\bar{g}}(\kappa,T)$.
\end{enumerate}
\end{proposition}
		
The proof of the above Proposition is given in Appendix \ref{app:A}.	
Applying it together with the Bretagnolle--Huber inequality \eqref{bh ineq} immediately yields the required lower bounds.
Since we aim at presenting them in a general version, we recall the classical notion of a loss function (see equation (2.3) in \cite{tsy09}), referring to a function $\ell\colon [0,\infty)\to[0,\infty)$ that is monotonically increasing and satisfies $\ell(0)=0$ and $\ell \nequiv 0$.
		
\begin{theorem}\label{thm:lowmarg}
Let everything be given as in Proposition \ref{prop: existence hyp}. 
Then, for any loss function $\ell$, there exist constants $c_{\ell,1},c_{\ell,2}>0$ such that, for large enough $T$, 
\begin{align*}
\inf_{\tilde{y}_T} \sup_{b\in\Sigma(\mathbf{C},A,\gamma)}\sup_{\bar{g}\in  \cG_b(y_1,\zeta,M,\Delta_0,n,\eta,\beta)}\sup_{g\in\cG_{\bar{g}}(\kappa,T)}\E_b\left[\ell\left(c_{\ell,1}T^{1/(2-2\beta)} \left(\Phi_b(g)-\frac{g(\tilde{y}_T)}{\xi_{b}(\tilde{y}_T)}\right)\right)\right]\geq c_{\ell,2},
\end{align*}
where the infimum extends over all estimators $\tilde{y}_T$.
\end{theorem}
\begin{proof}
Let $b,\bb,g,\bar{g}$ be given as in Proposition \ref{prop: existence hyp}. 
Then, \eqref{bh ineq} implies for large enough $T$ that, for any estimator $\tilde{y}_T$,
\begin{align*}
&\P_{b}\left(\Phi_b(g)-\frac{g(\tilde{y}_T)}{\xi_{b}(\tilde{y}_T)}>c_{L,3} T^{-1/(2-2\beta)}\right)+\P_{\bb}\left(\Phi_{\bb}(g)-\frac{g(\tilde{y}_T)}{\xi_{\bb}(\tilde{y}_T)}> c_{L,2}T^{-1/(2-2\beta)} \right)\\
&\quad\ge \P_{b}\left(\Phi_b(g)-\frac{g(\tilde{y}_T)}{\xi_{b}(\tilde{y}_T)}\leq c_{L,2} T^{-1/(2-2\beta)}\right)+\P_{\bb}\left(\Phi_{\bb}(g)-\frac{g(\tilde{y}_T)}{\xi_{\bb}(\tilde{y}_T)}> c_{L,2}T^{-1/(2-2\beta)} \right)			\\
&\quad\ge \frac{1}{2}\exp(-\KL(\Pro_b^T,\Pro_{\bb}^T))			\\
&\quad\ge \frac{1}{2}\exp(-c_{L,1}).
\end{align*}
For any given loss function $\ell$, there exists some $\alpha>0$ such that $\ell(\alpha)>0$. 
Hence, setting $c_L=c_{L,2}\land c_{L,3}$,
\begin{align}\notag \label{eq: loss bh}
&\E_b\left[\ell\left(\alpha\left(c_LT^{-1/(2-2\beta)} \right)^{-1}\left(\Phi_b(g)-\frac{g(\tilde{y}_T)}{\xi_{b}(\tilde{y}_T)}\right)\right)\right]  +\E_{\bb}\left[\ell\left(\alpha\left(c_LT^{-1/(2-2\beta)} \right)^{-1}\left(\Phi_{\bb}(g)-\frac{g(\tilde{y}_T)}{\xi_{\bb}(\tilde{y}_T)}\right)\right)\right]
\\\notag
&\quad\ge \ell(\alpha)\left(\P_{b}\left(\Phi_b(g)-\frac{g(\tilde{y}_T)}{\xi_{b}(\tilde{y}_T)}>c_{L,3} T^{-1/(2-2\beta)}\right)+\P_{\bb}\left(\Phi_{\bb}(g)-\frac{g(\tilde{y}_T)}{\xi_{\bb}(\tilde{y}_T)}> c_{L,2}T^{-1/(2-2\beta)} \right)\right)
\\
&\quad\ge \frac{\ell(\alpha)}{2}\exp(-c_{L,1}),
\end{align}
which concludes the proof as $2\max(x,y)\geq x+y$ for all $x,y\in\R$.
\end{proof}
		
\subsection{Lower bound for the general case}
The results of Section \ref{sec:upsimp} and Theorem \ref{thm:lowmarg} provide an optimal convergence rate for the simple regret under the margin condition. However, since this assumption is not always satisfied or verifiable, we now consider the more general situation with payoff functions from the class $\cG(y_1,\zeta,MM_2,\mathbf{C},A,\gamma)$.
The following result completes the upper bound \eqref{eq:upgen} with a corresponding lower bound, thus verifying optimality of the rate $T^{-1/2}$ in the general situation.
		
\begin{theorem}\label{thm: lower bound gen}
Let $a>0$, $0<M<a/2$, $0<y_1\leq a/2-M$, $\zeta\geq M+3a/2$, $\mathbf{C}\geq 1$, $\gamma>0$, $A>M+a+\gamma$ be given. 
Then, for any loss function $\ell$, there exist constants $c_{\ell,3},c_{\ell,4}>0$ such that, for large enough $T$, 
\begin{align*}
\inf_{\tilde{y}_T}\sup_{b\in\Sigma(\mathbf{C},A,\gamma)}\sup_{g\in\cG(y_1,\zeta,MM_2,\mathbf{C},A,\gamma)}\E_b\left[\ell\left(c_{\ell,3}T^{1/2} \left(\Phi_b(g)-\frac{g(\tilde{y}_T)}{\xi_{b}(\tilde{y}_T)}\right)\right)\right]\geq c_{\ell,4},
\end{align*}
where the infimum extends over all estimators $\tilde{y}_T$.
\end{theorem}
\begin{proof}
Fix $T>0$, denote $\ep=T^{-1/2}$ and, for $a>0$, set $A_0\coloneqq M+a$.
Define the hypotheses $b,\bb\colon \R\to\R$ by letting
\begin{equation*}
\begin{split}
b(x)\coloneqq\begin{cases}
						-x, & x\leq 0,\\
						0, & 0<x\leq A_0,\\
						-(x-A_0), &A_0<x,
					\end{cases}
					\quad \text{ and }\quad  
					\bb(x)\coloneqq\begin{cases}
						b(x), & x\leq a,\\
						-\ep(x-a), & a<x\leq A_0,\\
						-(x-A_0)-\ep(A_0-a), &A_0<x.
					\end{cases}
\end{split}
\end{equation*}
Then, for large enough $T$, it follows that $b,\bb\in\Sigma(\mathbf{C},A,\gamma)$.
Now, as in the proof of Proposition \ref{prop: existence hyp}, Girsanov's theorem gives, for large enough $T$,
\begin{align}\label{eq: kl bound gen}
\notag\KL(\Pro_{b}^T,\Pro_{\bb}^T)&=\E_b\left[\log\left(\frac{\rho_b}{\rho_{\bb}}(X_0)\right)\right]+\frac 1 2 \E_b\left[\int_0^T (b(X_s)-\bb(X_s))^2\d s\right]
				\\
				&\leq c+\ep^2\E_b\left[\int_0^T\1_{(a,\infty)}(X_s)\d s\right]\notag
				\\
				&\leq c,
\end{align} 
where $c>0$ is a strictly positive constant, whose value changes from line to line.
Furthermore, \eqref{eq: xib} implies for $0<x\leq A_0$ that
\begin{equation*}
\xi_b(x)=x^2+\sqrt{\pi}x.
\end{equation*}
Additionally, we obtain for $x\in(a,A_0)$ and large enough $T$,
\begin{align}\label{eq: lower bound gen}
\xi_{\bb}(x)\notag
&=\xi_b(a)+2\int_a^x \exp\left(\ep(y-a)^2\right)\left( \sqrt{\pi}/2+a+\int_{a}^y\exp\left(-\ep(z-a)^2\right)\d z\right)\d y
\\\notag
&=\xi_b(a)+2\int_0^{x-a} \exp\left(\ep y^2\right)
\left( \sqrt{\pi}/2+a+\int_{0}^{y}\exp\left(-\ep z^2\right)\d z\right)\d y
\\\notag&\geq \xi_b(a)+2\int_0^{x-a} \left(1+\ep y^2\right)
\left( \sqrt{\pi}/2+a+y-\frac{\ep}{3}y^3\right)\d y
\\\notag&=\xi_b(a)+(\sqrt{\pi}+a+x)(x-a)+\frac{\ep}{3}(x-a)^3(\sqrt\pi+2a)+\frac{\ep}{3}(x-a)^4-\frac{\ep^2}{9}(x-a)^6
\\
&\geq \xi_b(x)+\frac{\ep}{2}(x-a)^3.
\end{align}
Now, for $0<M<a/2$, $\delta>0$, define $f,g\colon(0,\infty)\to \R$ by letting
\[	f(y)\coloneqq \begin{cases}
				M-\vert y-a/2\vert, &0<x<a\\
				M-\delta-\vert y-3a/2-\delta\vert,&x\geq a
			\end{cases}
			\quad\mathrm{ and }\quad
			g(y)\coloneqq f(y)\xi_{\bb}(y),\]
			respectively, which implies $g\in \cG(y_1,\zeta,MM_2,\mathbf{C},A,\gamma)$ for large enough $T$.
Then, for $0<x<a$,
\[\frac{g}{\xi_b}(x)=\frac{g}{\xi_{\bb}}(x), \]
and for small enough $\delta$, \eqref{eq: lower bound gen} implies
\begin{align*}
\frac{g}{\xi_b}(3a/2+\delta)&=(M-\delta)\frac{\xi_{\bb}(3a/2+\delta)}{\xi_b(3a/2+\delta)}\\
&\geq (M-\delta)\left(1+\ep\frac{(a/2+\delta)^{3}}{2(3a/2+\delta)^2+\sqrt{\pi}(3a+2\delta)}\right)
\\
				&\geq (M-\delta)\left(1+\ep\frac{a^2}{64a+32\sqrt{\pi}}\right).
			\end{align*}
			Hence, setting
			\[c_a\coloneqq \frac{a^2}{64a+32\sqrt{\pi}}\quad \text{ and } \quad \delta\coloneqq \frac{Mc_a\ep}{2}\]  gives, for large enough $T$,
			\begin{align*}
				\frac{g}{\xi_b}(3a/2+\delta)&\geq (M-\delta)\left(1+c_a\ep\right)
				\\&= M+Mc_a\ep-\delta+c_a\ep\delta
				\\&\geq M+\delta.
			\end{align*}
			Thus,  \eqref{eq: kl bound gen} and the Bretagnolle--Huber inequality imply for any estimator $\tilde{y}_T$
			\begin{align*}
				\P_b\left(\Phi_b(g)-\frac{g}{\xi_b}(\tilde{y}_T)\geq \delta\right)+\P_{\bb}\left(\Phi_{\bb}(g)-\frac{g}{\xi_{\bb}}(\tilde{y}_T)\geq \delta\right) 
				&\geq \P_b\left(\tilde{y}_T<a\right)+\P_{\bb}\left(\tilde{y}_T\geq a\right) 
				\\
				&\geq c,
\end{align*}
for some strictly positive constant $c$. 
Arguing as in the derivation of \eqref{eq: loss bh} concludes the proof.
		\end{proof}

\section{Implications for the cumulative regret}\label{sec: cum}
In this section, we analyse the significance of the results of Sections \ref{sec: upper bound} and \ref{sec: lower bound} for the cumulative regret. In particular, we examine their implications for the exploration-exploitation strategy presented in \cite{chris23}, before deriving minimax lower bounds for the cumulative regret of data-driven impulse control strategies.

\subsection{Upper bounds on the cumulative regret}\label{subsec: upp cum}
The results of Section \ref{sec:upsimp} address the issue of \emph{exploration}, i.e., the identification of an estimator of the unknown optimal threshold. 
However, a data-driven control strategy must also consider the \emph{exploitation aspect} at the same time, since we want to maximize our expected reward. Balancing the costs of exploration and exploitation in an appropriate way was one of the main contributions of \cite{chris23}, see also \cite[Section 2.2.3]{chris23+}. 
The basic idea, which is also applicable here, is to choose a strategy that gives us a subpath of length $S_T$ up to time $T$ (exploration periods), on which we observe the unstopped process, see Figure \ref{fig: datadriven}.  
  \begin{figure}[h!]
			\centering
			\includegraphics[width=\linewidth]{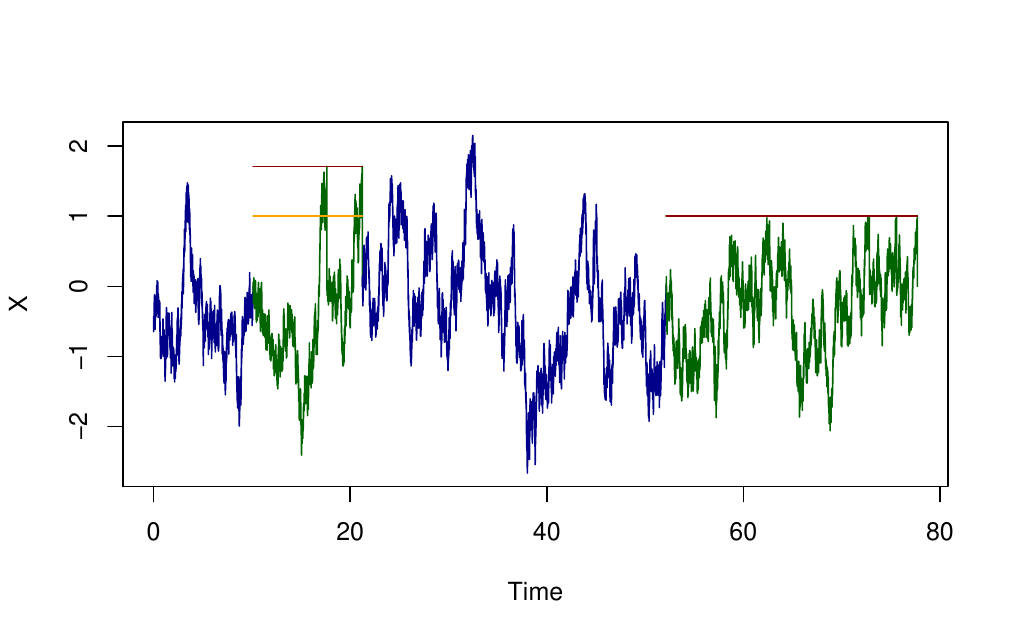}
			\caption{Plot of a process controlled by the proposed data-driven strategy. The \textcolor{darkblue}{exploration periods} are marked blue and the \textcolor{darkgreen}{exploitation periods} are green. The \textcolor{dex}{estimated barrier} is given in red, whereas the \textcolor{orange}{true optimal barrier} is orange.}
			\label{fig: datadriven}
		\end{figure}
On this subpath, we can then apply the estimation procedure studied in Section \ref{sec:upsimp}. 
However, since observing the process without stopping is suboptimal for the underlying stopping problem, this incurs a cost of order $S_T$. The rest of the time (exploitation periods), we try to stop the process at the optimal boundary. For this purpose, information about the stopped process of length $S_t$ is available at time $t$. Now, if the estimation procedure for the static regret yields the order $O(\psi(S_t))$, this leads to cumulative costs of order $\int\psi(S_t) \d t$ up to $T$, where we integrate over the exploitation periods until $T$. To implement the procedure, one now balances both costs.
For example, for margin parameters $\beta\in(0,1)$, by Theorem \ref{thm: upper bound}, we are in the case $\psi(t)=t^{-\frac{1}{2-2\beta}}$, so the choice $S_T=T^{\frac{2-2\beta}{3-2\beta}}$ leads to balancing of the costs and thus to a cumulative regret of this same order.

\subsection{Lower bound for the cumulative regret}
Similar to the bandit context (see, for example, Corollary 1 of \cite{loca18}, or Section 3 of \cite{bubetal11}), we can apply our lower bounds for the simple regret to obtain minimax lower bounds for the cumulative regret. Before stating this result, we recall that, given an increasing sequence of stopping times $K=(\tau_n)_{n\in\N}$ satisfying $\lim_{n\to\infty}\tau_n=\infty$, the controlled process $\X^K=(X^K_t)_{t\geq0}$ is defined to follow the dynamics \eqref{eq: sde} for $\tau_n\leq t<\tau_{n+1}$ and satisfies $X^K_{\tau_n}=0,$ for any $n\in\N$. We now introduce the notion of \emph{impulse control strategies}, which can be interpreted as data-driven control strategies related to the problem under consideration. 

\begin{definition}
A sequence of stopping times $K=(\tau_n)_{n\in\N_0}$ is called an \emph{impulse control strategy wrt $\cG(y_1,\zeta,M,\mathbf{C},A,\gamma)$} if
\[\tau_0=0,\quad \tau ^K_n\coloneqq\inf\{t\geq \tau_{n-1}: X_t\geq y_n \},\quad \forall n \in\N, \]
where each $y_n$ is an $\cF_{\tau_{n-1}}$-measurable random variable such that $y_n\in[y_1,\zeta]$ a.s., for $(\cF_t)_{t\geq0}$ denoting the natural filtration of $\X^K$. 
We call $K$ a \emph{general impulse control strategy} if there exists a set $\cG(y_1,\zeta,M,\mathbf{C},A,\gamma)$ such that $K$ is an impulse control strategy wrt $\cG(y_1,\zeta,M,\mathbf{C},A,\gamma)$.
 \end{definition}

For each impulse control strategy $K$, the expected cumulative regret up to time $T$ wrt a payoff function $g\in \cG(y_1,\zeta,M,\mathbf{C},A,\gamma)$ and a drift $b\in\sigma(\mathbf{C},A,\gamma)$ can then be written as 
  \[\Phi_b(g)T- \E_b\left[\sum_{n:\tau_n\leq T} g(X^K_{\tau_{n-}}) \right],\]
which corresponds to the expected loss of following the impulse control strategy $K$ instead of the optimal strategy until time $T$.

\begin{theorem} \label{thm: lower bound cum}
Let everything be given as in Theorem \ref{thm: lower bound gen}. 
Then, there exists a constant $c_{\ell,5}>0$ such that, for large enough $T>0$,
\[\inf_K\sup_{b\in\Sigma(\mathbf{C},A,\gamma)}\sup_{g\in\cG(y_1,\zeta,MM_2,\mathbf{C},A,\gamma)}\left(\Phi_b(g)T- \E\left[\sum_{n:\tau_n\leq T} g(X^K_{\tau_{n-}}) \right]\right)\geq c_{\ell,5}\sqrt{T},  \]
where the infimum extends over all impulse control strategies wrt $\cG(y_1,\zeta,MM_2,\mathbf{C},A,\gamma)$.
\end{theorem}

The proof of Theorem \ref{thm: lower bound cum} can be found in Appendix \ref{app:A}. We see that the lower bound for the cumulative regret given in the previous theorem is of the form $T\times\psi(T)$, where $\psi(T)=T^{-1/2}$ denotes the lower bound for the simple regret established in Theorem \ref{thm: lower bound gen}. This is not a surprising result, since the payoff per time in each period can be bounded from below by the lower bound on the simple regret. Lower bounds for the cumulative regret of the same form are also common in the bandit literature, see e.g.\ Corollary 1 in \cite{loca18}. Moreover, we note that using the results of Theorem \ref{thm:lowmarg}, similar minimax lower bounds of the same form (i.e., $T\times T^{-1/(2-2\beta)}=T^{(1-2\beta)/(2-2\beta)}$) can be obtained for the cumulative regret of data-driven impulse control strategies under the margin condition. 
However, we omit the rigorous proof of this result as it does not provide much further insight.

\section{Concluding remarks and simulation study}\label{sec: fin}
It is well known that there are hardly any theoretical guarantees for the multitude of RL algorithms used in practice. 
For a brief overview of the current state of the art, we refer to Section 1.2 in \cite{chris23++}.
In order to at least gain an understanding of the underlying principles of decision making under uncertainty, which are not transparent in the general setting of Markov decision processes, multi-armed bandits are often studied as a simplified model.
This approach is also used, for example, in the standard textbook treatment \cite{sutbar18}.
While there is an extensive literature with precise results for multi-armed bandits, in contrast to many unproven methods for the general Markov decision situation, it is also clear that the stochastic bandit model has limitations in reflecting real-world scenarios.
Extending the standard assumptions (such as moving from an i.i.d. setting to the more realistic assumption of dependent reward distributions) remains a challenge. 
We propose the alternative approach of analysing stochastic control problems with a nonparametric statistical structure. Basic principles (such as the construction of data-based strategies and the analysis of the exploration-exploitation problem) were clarified in the initial work \cite{chris23}. The follow-up paper \cite{chris23+} provides a simultaneous analysis of stochastic control problems for diffusion and Lévy processes, thus confirming the robustness of the underlying data-based control approach. The recent preprint \cite{chris23++} extends the proposed principle to the context of multivariate reflection problems. While the aforementioned studies can be understood primarily as proofs of concept and do not address issues such as optimality, the focus of this paper is on the statistical aspects of the stochastic control problem.
This motivates the study of the pure exploration problem of minimising the simple regret, which we solve by providing an estimator of the optimal stopping barrier that is rate-optimal in a variety of settings. Furthermore, our analysis allows us to present uniform, non-asymptotic PAC bounds for the simple regret, enabling a decision maker to decide how much data is needed to make a reasonable stopping decision.

\paragraph{Future research}
The present analysis provides a fundamental basis for the investigation of the cumulative regret, where a number of open questions remain. First, the optimal choice of exploration time $S_T$ in the exploration-exploitation strategy introduced in Section \ref{subsec: upp cum} is given as $S_T=T^{\frac{2-2\beta}{3-2\beta}}$ under the margin condition, and thus depends on the generally unknown parameter $\beta$. Therefore, it could be interesting to construct an adaptive procedure that asymptotically achieves the selection of the optimal stopping barrier, together with an optimal choice of exploration time. Furthermore, our approach relies on an indirect estimation of $\xi_b$ through the invariant density and the corresponding distribution function. It might also be worthwhile to investigate a more direct approach by directly estimating the expected hitting times. This could also pave the way for novel strategies specifically designed for minimising the cumulative regret, which is particularly important since, at least in the bandit setting (see Section 3 in \cite{bubetal11}), optimising the simple or the cumulative regret are fundamentally different tasks requiring different policies. For example, one could investigate strategies that do not split the process into exploration and exploitation periods, but use the data from each exploitation cycle to directly estimate $\xi_b$ and thus the optimal stopping barrier. 
Following the idea of UCB strategies (see Chapters 7 to 10 in \cite{latti20}), a summand depending on the uncertainty could be added to the estimated optimal stopping barrier $\hat{y}$, ensuring that not only the values of $\xi_b$ up to $\hat{y}$ can be estimated. However, since this approach in the bandit literature is based on the assumption that the distributions of the arms are sub-Gaussian, such a result seems to be out of reach at the moment, especially since the distributions of the first hitting times are generally unknown. Finally, a more thorough investigation of a lower bound for the cumulative regret under the margin condition might be interesting, aiming at results similar to Theorem 3 in \cite{loca18}. Intuitively, this theorem shows that a strategy that achieves an optimal cumulative regret over a class of regular payoff functions cannot achieve optimal results over a class of more irregular payoff functions. We conjecture that analogous results are also possible in the given setting, corresponding to showing that if a strategy achieves an optimal regret for all payoff functions satisfying the margin condition with parameter $\beta_1\in(0,1)$, it cannot uniformly achieve the optimal cumulative regret over all payoff functions with regularity $\beta_2\in(0,1)$ if $\beta_1>\beta_2$.

\paragraph*{Simulation study}
To gain insight into our theoretical results in a more practical setting, we perform a brief simulation study. 
For this, we simulate an Ornstein--Uhlenbeck process with drift $b(x)=-x/2$, using an Euler--Maruyama scheme with step size $10^{-2}$ and time horizon $T$ in \texttt{R}. 
Note that this particular choice of drift function implies that the invariant distribution is standard normal. 
\begin{figure}[h!]
			\includegraphics[width=.5\linewidth]{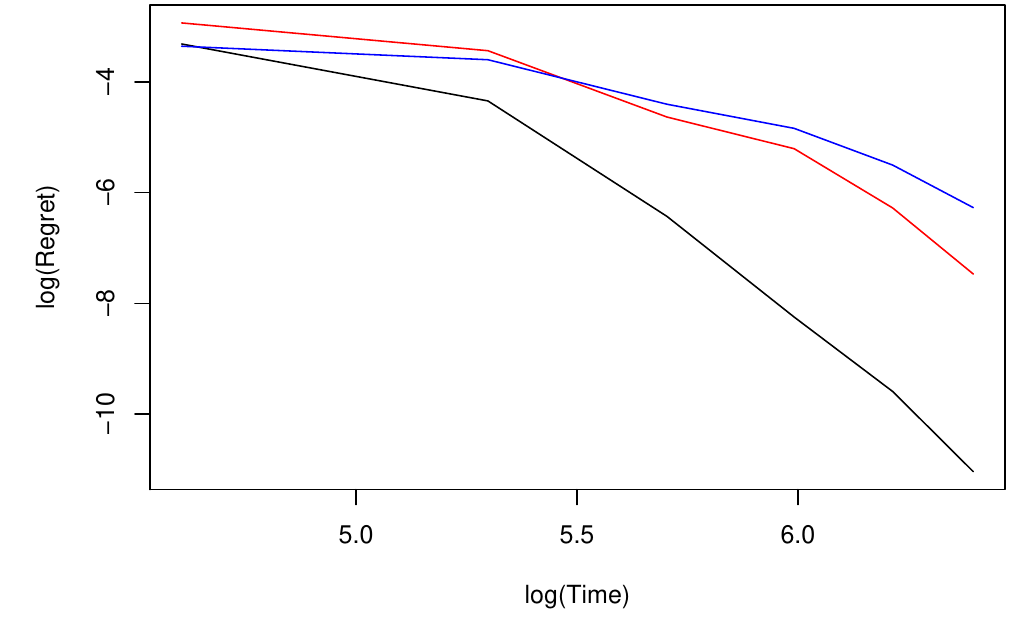}
			\hspace*{-0.6em}
			\includegraphics[width=.5\linewidth]{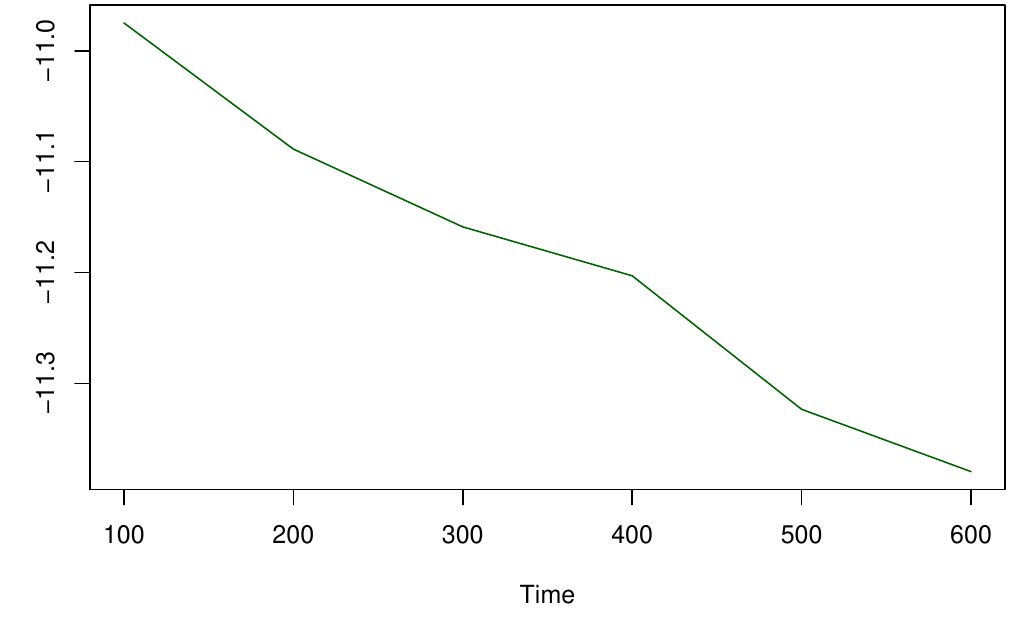}
			\caption{Plot of the logarithm of the simple regret for $\beta=\textcolor{blue}{0.25},\textcolor{red}{0.5},0.75,$ against $\log(T)$ on the left and plot of the logarithm of the simple regret against $T$ for $\beta=\textcolor{WIMgreen}{1}$ on the right.}
   \label{fig: sim1}
		\end{figure}
		As payoff functions, we choose
		\[g(x)=\left(1-\vert 1-x\vert^{1/\beta}\right)\xi_b(x),\quad x>0, \]
		with $\beta\in\{0.25,0.5,0.75,1\}$. 
Plots of these functions are given in Figure \ref{fig: payoff}.
For estimating the quantities of interest, we employ the commands \texttt{ecdf} and \texttt{density} from the \texttt{R} package \texttt{stats}, with bandwidth $T^{-1/2}$ and the Epanechnikov kernel for the latter. 
In the cases where $\beta<1$, we choose $T$ on a logarithmic grid ranging from $\e^3$ to $\e^8$.
For $\beta=1$, we consider $T\in \{100,200,\ldots,600\}$. 
For each time horizon and choice of $\beta$, the estimation procedure is then repeated $50$ times. 
The results of this simulation study can be seen in Figure \ref{fig: sim1}, where the logarithm of the mean simple regret, corresponding to $\vert \hat{y}_T-1\vert^{1/\beta}$ in the given setting, is plotted against $\log(T)$ for $\beta<1$, respectively $T$ itself for $\beta=1$.
We see that the slope gets steeper as $\beta$ increases on the left-hand side and also that the logarithm of the simple regret decreases linearly for $\beta=1$, which is both in accordance with our theoretical results in Theorem \ref{thm: upper bound}.

		\appendix
		
		\section{Remaining proofs}\label{app:A}
		This appendix contains the proof of three crucial auxiliary results and of the lower bound for the cumulative regret stated in Theorem \ref{thm: lower bound cum}.
		We start by providing the proof of the moment bounds for the estimators of the invariant density and distribution function, respectively.
		
		\begin{proof}[Proof of part (a) of Proposition \ref{prop: est mom}]
			Throughout the proof, we denote by a slight abuse of notation $\Sigma=\Sigma(\mathbf{C},A,\gamma)$.
			Let $p\ge 1$.
			Firstly, we note that, for any $x\in\R$, $b\in\Sigma$,
			\begin{align*}
				&\E_b\left[\big\vert\hat{\rho}_T(x)-\rho_b(x)\big\vert^p\right]\\
				&\quad\le 2^{p-1}\Bigg(\E_b\left[\left\vert\frac{1}{T}\int_0^T K_T(x-X_s)\d s-\E_b[ K_T(x-X_0)\right\vert^p \right]+\left\vert\E_b\left[ K_T(x-X_0)\right]-\rho_b(x)\right\vert^p \Bigg).
			\end{align*}
			Since $b\in\Sigma$ implies $\rho_b\in\cH(1,\cL)$ for some constant $\cL,$ classic arguments (see Proposition 1.2 in \cite{tsy09}) give, since $K$ is bounded with compact support,
			\begin{equation*}
				\sup_{b\in\Sigma}\left\vert\E_b\left[ K_T(x-X_0)\right]-\rho_b(x)\right\vert^p
				\le\left(\cL T^{-1/2}\int \vert u K(u)\vert \d u \right)^p.
			\end{equation*}
			Furthermore, Proposition 7 in \cite{aeck21} gives that there exists a constant $c_1>0$ such that
			\begin{align*}
				&\sup_{b\in\Sigma}\E_b\left[\left\vert\frac{1}{T}\int_0^T K_T(x-X_s)\d s-\E_b[ K_0(x-X_s)]\right\vert^p\right]\\
				&\quad=\sup_{b\in\Sigma}\E_b\left[\left\vert\frac{1}{\sqrt{T}}\int_0^T K((x-X_s)T^{1/2})-\E_b[ K((x-X_0)T^{1/2})]\d s\right\vert^p\right]\\
				&\quad\le c_1^p T^{-p/2}\left( p^{p/2}+ p^p T^{-p/2} \right),
			\end{align*}
			completing the proof of the first claim.
			For the second assertion, applying Itô's formula as in \cite{kuto01} shows that, for any $x\in\R$,
			\begin{equation*}
				\hat{F}_T(x)-F_b(x)
				=\frac 2 T \left(\int_{X_0}^{X_T} \frac{F_b(u\land x)-F_b(u)F_b(x)}{\rho_b(u)}\d u-\int_0^T \frac{F_b(x\land X_s)-F_b(x)F_b(X_s)}{\rho_b(X_s)}\d W_s \right).
			\end{equation*}
			Now note that, for any $u,x\in\R$,
			\begin{align*}
				F_b(u\land x)-F_b(u)F_b(x)&=\begin{cases}
					F_b(x)(1-F_b(u)), \quad &x\leq u,\\
					F_b(u)(1-F_b(x)), \quad &x>u
				\end{cases}\\
				&\le F_b(u)(1-F_b(u)),
			\end{align*}
			implying that
			\begin{align}
				\frac{F_b(u\land x)-F_b(u)F_b(x)}{\rho_b(u)}&\le \left(\sup_{z\geq0}\frac{1-F_b(z)}{\rho_b(z)}\lor\sup_{z\leq 0}\frac{F_b(z)}{\rho_b(z)}\right)\notag\\
				&\le c_F\label{eq: cdf proof},
			\end{align}
			where $c_F>0$ is a constant depending on the class $\Sigma$ (see the proof of Proposition 7 in \cite{aeck21}). 
			Hence, an application of equation (4.4) in \cite{aeck21} gives 
			\begin{align*}
				\sup_{b\in\Sigma}\E_b\left[\left\vert\int_{X_0}^{X_T} \frac{F_b(u\land x)-F_b(u)F_b(x)}{\rho_b(u)}\d u\right\vert^p \right]
				&\le c_F^p\sup_{b\in\Sigma}\E_b[\vert X_T-X_0\vert^p]\\
				&\le c_F^p2^p\sup_{b\in\Sigma}\E_b[\vert X_0\vert^p]\\
				&\le c_{\mathrm{mom}}^pc_F^p2^pp^p,
			\end{align*}
			where $c_{\mathrm{mom}}>0$ is again a constant depending only on the class $\Sigma$. 
			Applying \eqref{eq: cdf proof} together with the Burkholder--Davis--Gundy inequality then also yields for a constant $c_p>0$, depending on $p$,
			\begin{align*}
				&\sup_{b\in\Sigma}\E_b\left[\left\vert\int_0^T \frac{F_b(x\land X_s)-F_b(x)F_b(X_s)}{\rho_b(X_s)}\d W_s\right\vert^p\right]\\
				&\quad\leq c_p \sup_{b\in\Sigma}\E_b\left[\left\vert\int_0^T \left(\frac{F_b(x\land X_s)-F_b(x)F_b(X_s)}{\rho_b(X_s)}\right)^2\d s\right\vert^{p/2}\right]\\
				&\quad\leq c_p c_F^p T^{p/2}.
			\end{align*}
			Since Proposition 4.2 in \cite{barlow82} gives that $c_p\leq c^p p^{p/2}$ for some universal constant $c>0$, the proof of the second assertion is complete. 
		\end{proof}

			\begin{proof}[Proof of Lemma \ref{lemma: local upper bound}]
				Throughout the proof, we denote 
				\[
				\bar{\rho}_T(x)\coloneqq \hat{\rho}_T(x)\lor a, \quad x\in\R.
				\] 
				Let $p\ge1$, $b\in\Sigma(\mathbf{C},A,\gamma)$, $g\in \cG_b(y_1,\zeta,M,\Delta_0,n,\eta,\beta)$, and $0<\Delta<\Delta_0$ be given, and
				recall that
				\[
				\mathcal E(\Delta)=\left\{y\in[y_1,\zeta]\colon\Phi_b(g)-\frac{g}{\xi_b}(y)\le \Delta\right\}.
				\]
				Then, by the margin condition, there exist points $y_1^*,\ldots,y_n^*\in(y_1,\zeta)$ such that
				\begin{align}\nonumber
					&\E_b\left[ \left\vert\Phi_b(g)-\frac{g}{\xi_b}(\hat{y}_T) \right\vert^p\1_{\mathcal E(\Delta)}(\hat{y}_T)  \right]
					\\\nonumber
					&\quad\le \sum_{i=1}^{n}\E_b\left[ \Big\vert\Phi_b(g)-\frac{g}{\xi_b}(\hat{y}_T) \Big\vert^p\1_{\mathcal E(\Delta)\bigcap (y_i^*-\eta\Delta^\beta/2,y_i^*+\eta\Delta^\beta/2)}(\hat{y}_T)\right]
					\\\nonumber
					&\quad\le\sum_{i=1}^{n}\E_b\left[ \Big\vert\Phi_b(g)-\frac{g}{\xi_b}(\hat{y}_T)+\frac{g}{\hat{\xi}_T}(\hat{y}_T)-\frac{g}{\hat{\xi}_T}(y_i^*) \Big\vert^p\1_{\mathcal E(\Delta)\bigcap (y_i^*-\eta\Delta^\beta/2,y_i^*+\eta\Delta^\beta/2)}(\hat{y}_T)\right]
					\\\nonumber
					&\quad\le\sum_{i=1}^{n}\E_b\left[\sup_{y\in\mathcal E(\Delta)\colon  \vert y_i^*-y\vert\leq \eta\Delta^\beta} \Big\vert\Phi_b(g)-\frac{g}{\xi_b}(y)+\frac{g}{\hat{\xi}_T}(y)-\frac{g}{\hat{\xi}_T}(y_i^*)\Big\vert^p  \right]
					\\\nonumber
					&\quad= \sum_{i=1}^{n}\E_b\left[\sup_{y\in\mathcal E(\Delta)\colon  \vert y_i^*-y\vert\leq \eta\Delta^\beta} \Big\vert g(y)\frac{\xi_b(y)-\hat{\xi}_T(y)}{\xi_b(y)\hat{\xi}_T(y)}
					-g(y_i^*)\frac{\xi_b(y_i^*)-\hat{\xi}_T(y_i^*)}{\xi_b(y_i^*)\hat{\xi}_T(y_i^*)}  \Big\vert^p\right]\eqqcolon \sum_{i=1}^n\mathscr T_i.
				\end{align}
				Now, for any $i\in\{1,\ldots,n\}$, it holds
				\begin{align}\nonumber
					\mathscr T_i&\le
					2^{p-1}\E_b\left[\sup_{y\in\mathcal E(\Delta)\colon  \vert y_i^*-y\vert\leq \eta\Delta^\beta} \Big \vert g(y)\frac{\xi_b(y)-\hat{\xi}_T(y)}{\xi_b(y)\hat{\xi}_T(y)}-g(y_i^*)\frac{\xi_b(y)-\hat{\xi}_T(y)}{\xi_b(y_i^*)\hat{\xi}_T(y)} \Big\vert^p  \right]
					\\\nonumber
					&\quad+2^{p-1}\E_b\left[\sup_{y\in\mathcal E(\Delta)\colon  \vert y_i^*-y\vert\leq \eta\Delta^\beta} \Big \vert g(y_i^*)\frac{\xi_b(y)-\hat{\xi}_T(y)}{\xi_b(y_i^*)\hat{\xi}_T(y)}-g(y_i^*)\frac{\xi_b(y_i^*)-\hat{\xi}_T(y_i^*)}{\xi_b(y_i^*)\hat{\xi}_T(y_i^*)} \Big\vert^p \right]
					\\\nonumber
					&\le 2^{2p-1}M_1^{-p}\E_b\left[\sup_{y\in\mathcal E(\Delta)\colon  \vert y_i^*-y\vert\leq \eta\Delta^\beta}\Big\vert\Phi_b(g)-\frac{g(y)}{\xi_b(y)} \Big\vert^p 
					\left \vert \xi_b(y)-\hat{\xi}_T(y) \right\vert^p  \right]
					\\\nonumber
					&\quad+2^{3p-1}M_1^{-2p} M^p
     \\&\qquad\nonumber\times
					\E_b\left[\sup_{y\in\mathcal E(\Delta)\colon  \vert y_i^*-y\vert\leq \eta\Delta^\beta}\Big \vert \big(\xi_b(y)-\hat{\xi}_T(y)\big)\hat{\xi}_T(y_i^*)-\big(\xi_b(y_i^*)-\hat{\xi}_T(y_i^*)\big)\hat{\xi}_T(y) \Big\vert^p  \right]
					\\\label{eq: long comp 00}
					&\le2^{2p-1}M_1^{-p}\Delta^p\E_b\left[\sup_{y\in[y_1,\zeta]\colon\vert y_i^*-y\vert\leq \eta\Delta^\beta} \left \vert \xi_b(y)-\hat{\xi}_T(y) \right\vert^p  \right]+2^{4p-2}M_1^{-2p} M^p\left(\mathscr T_{i,1}+\mathscr T_{i,2}\right)
				\end{align}
				with
				\begin{align}\nonumber
					\mathscr T_{i,1}&\coloneqq \E_b\left[\sup_{y\in\mathcal E(\Delta)\colon  \vert y_i^*-y\vert\leq \eta\Delta^\beta} \Big \vert\big(\xi_b(y_i^*)-\hat{\xi}_T(y_i^*)\big)\big(\hat{\xi}_T(y_i^*)-\hat{\xi}_T(y)\big) \Big\vert^p  \right]\\\label{eq: long comp 0}
					&\ \ \le \eta^p\Delta^{\beta p}a^{-p}\E_b\left[ \left \vert\xi_b(y_i^*)-\hat{\xi}_T(y_i^*)\right\vert^p  \right],\\\nonumber
					\mathscr T_{i,2}&\coloneqq \E_b\left[\sup_{y\in\mathcal E(\Delta)\colon  \vert y_i^*-y\vert\leq \eta\Delta^\beta}\Big\vert\big(\xi_b(y)-\hat{\xi}_T(y)-\xi_b(y_i^*)+\hat{\xi}_T(y_i^*)\big)\hat{\xi}_T(y_i^*) \Big\vert^p  \right]\\\label{eq: long comp 1}
					&\ \ \le \zeta^pa^{-p}\E_b\left[\sup_{y\in[y_1,\zeta]\colon \vert y_i^*-y\vert\leq \eta\Delta^\beta} \left\vert \xi_b(y)-\hat{\xi}_T(y)-\xi_b(y_i^*)+\hat{\xi}_T(y_i^*) \right\vert^p  \right],
				\end{align}
				where we used the condition $\Phi_b(g)-g(y)/\xi_b(y)\leq \Delta$ and the Lipschitz continuity of $\hat{\xi}_T$, combined with $\vert y^*_i-y\vert\leq \eta\Delta^\beta$. 
				Furthermore using that $\max(a,b)=\tfrac12(a+b+\vert a-b\vert)$ for any $a,b\in\R$, we get that, for any $y\in[y_1,\zeta]$, 
				\begin{align*}
					& \left\vert \xi_b(y)-\hat{\xi}_T(y)-\xi_b(y_i^*)+\hat{\xi}_T(y_i^*) \right\vert^p \\
					&= \left\vert \int_y^{y^*}\left(\frac{\hat{F}_T(z)}{\bar{\rho}_T(z)}-\frac{2F_b(z)}{\rho_b(z)}\right)\d z-\frac 1 2 \left\vert2\int_0^y\frac{\hat{F}_T(z)}{\bar{\rho}_T(z)}\d z-\frac{M_1}{2}\right\vert+\frac 1 2 \left\vert2\int_0^{y^*}\frac{\hat{F}_T(z)}{\bar{\rho}_T(z)}\d z-\frac{M_1}{2}\right\vert \right\vert^p
					\\
					&\le2^p\Bigg(2^p\left\vert \int_y^{y^*}\frac{\hat{F}_T(z)}{\bar{\rho}_T(z)}-\frac{F_b(z)}{\rho_b(z)}\d z \right\vert^p\\&\hspace*{5em}+\left\vert \int_y^{y^*}\frac{\hat{F}_T(z)}{\bar{\rho}_T(z)}\d z+\frac 1 2 \left\vert2\int_0^y\frac{\hat{F}_T(z)}{\bar{\rho}_T(z)}\d z-\frac{M_1}{2}\right\vert-\frac 1 2 \left\vert2\int_0^{y^*}\frac{\hat{F}_T(z)}{\bar{\rho}_T(z)}\d z-\frac{M_1}{2}\right\vert \right\vert^p\Bigg),
				\end{align*}
implying that
\begin{align*}
&\E_b\left[\sup_{y\in[y_1,\zeta]\colon \vert y_i^*-y\vert\leq \eta\Delta^\beta} \left\vert \xi_b(y)-\hat{\xi}_T(y)-\xi_b(y_i^*)+\hat{\xi}_T(y_i^*) \right\vert^p  \right]	\\
&\quad\le2^p\Bigg(2^p\E_b\left[ \left(\int_{y_i^*-\eta\Delta^\beta}^{y_i^*+\eta\Delta^\beta} \left\vert\frac{\hat{F}_T(z)}{\bar{\rho}_T(z)}-\frac{F_b(z)}{\rho_b(z)} \right\vert\d z\right)^p  \right]+\E_b\Bigg[\1\left\{2\int_0^{y_1}\frac{\hat{F}_T(z)}{\bar{\rho}_T(z)}\d z\leq \frac{M_1}{2}\right\}\\
&\quad\qquad\times\sup_{y\in[y_1,\zeta]\colon \vert y_i^*-y\vert\leq \eta\Delta^\beta} \left\vert \int_y^{y^*}\frac{\hat{F}_T(z)}{\bar{\rho}_T(z)}\d z+\frac 1 2 \left\vert2\int_0^y\frac{\hat{F}_T(z)}{\bar{\rho}_T(z)}\d z-\frac{M_1}{2}\right\vert-\frac 1 2 \left\vert2\int_0^{y^*}\frac{\hat{F}_T(z)}{\bar{\rho}_T(z)}\d z-\frac{M_1}{2}\right\vert \right\vert^p  \Bigg]\Bigg)
					\\
&\quad\leq 2^{2p}\Bigg(\E_b\left[ \left(\int_{y_i^*-\eta\Delta^\beta}^{y_i^*+\eta\Delta^\beta} \left\vert\frac{\hat{F}_T(z)}{\bar{\rho}_T(z)}-\frac{F(z)}{\rho(z)} \right\vert\d z\right)^p  \right]+2^pM_1^{-p}a^{-p} \eta^p \Delta^{\beta p}\E_b\left[\big\vert \xi_b(y_1)-\hat{\xi}_T(y_1)\big\vert^p\right]\Bigg)
					\\
&\quad\leq 2^{3p}\eta^p\Delta^{\beta p}\Bigg((2\eta\Delta^\beta)^{-1}\E_b\left[ \int_{y_i^*-\eta\Delta^\beta}^{y_i^*+\eta\Delta^\beta} \left\vert\frac{\hat{F}_T(z)}{\bar{\rho}_T(z)}-\frac{F(z)}{\rho(z)} \right\vert^p\d z \right]+M_1^{-p}a^{-p} \E_b\left[\big\vert \xi_b(y_1)-\hat{\xi}_T(y_1)\big\vert^p\right]\Bigg)\stepcounter{equation}\tag{\theequation}\label{eq: long comp 2},
\end{align*}
where we again used Lipschitz continuity of $\hat{\xi}_T$, together with the reverse triangle inequality, Markov's inequality and Hölder's inequality.
Now Proposition \ref{prop: est mom} gives
\begin{align*}
&\E_b\left[\sup_{y\in[y_1,\zeta]\colon\vert y_i^*-y\vert\leq \eta\Delta^\beta} \big \vert \xi_b(y)-\hat{\xi}_T(y) \big\vert^p  \right]\lor \E_b\left[\big\vert \xi_b(y_1)-\hat{\xi}_T(y_1)\big\vert^p\right]
					\\
					&\quad\le \E_b\left[\sup_{y\in[y_1,\zeta]} \left \vert \xi_b(y)-\hat{\xi}_T(y) \right\vert^p  \right]
					\\
					&\quad\le2^p \E_b\left[\sup_{y\in[y_1,\zeta]} \bigg \vert \int_0^y\frac{\hat{F}_T(z)}{\bar{\rho}_T(z)}-\frac{F_b(z)}{\rho_b(z)}\d z \bigg\vert^p  \right]
					\\
					&\quad\le 2^p\zeta^{p-1} \E_b\left[\int_0^\zeta\bigg\vert \frac{\hat{F}_T(z)}{\bar{\rho}_T(z)}-\frac{F_b(z)}{\rho_b(z)} \bigg\vert^p\d z  \right]
					\\
					&\quad\le 2^{2p-1}\zeta^{p-1} \int_0^\zeta\E_b\left[\bigg\vert \frac{\hat{F}_T(z)}{\bar{\rho}_T(z)}-\frac{\hat{F}_T(z)}{\rho_b(z)}\bigg\vert^p+\bigg\vert\frac{\hat{F}_T(z)}{\rho_b(z)}-\frac{F_b(z)}{\rho_b(z)} \bigg\vert^p  \right]\d z
					\\
					&\quad\le 2^{2p-1}\zeta^{p-1} \int_0^\zeta\E_b\left[a^{-2p}\left \vert \bar{\rho}_T(z)-\rho_b(z)\right\vert^p+a^{-p}\big\vert\hat{F}_T(z)-F_b(z) \big\vert^p  \right]\d z
					\\
					&\quad\le 2^{2p-1}\zeta^{p-1} \int_0^\zeta\E_b\left[a^{-2p}\big \vert \hat{\rho}_T(z)-\rho_b(z)\big\vert^p+a^{-p}\big\vert\hat{F}_T(z)-F_b(z) \big\vert^p  \right]\d z
					\\
					&\quad\le 2^{2p-1}\zeta^{p} (a^{-2p}c_1^p +a^{-p}c_2^p)T^{-p/2}\left( p^{p/2}+ p^p T^{-p/2} \right)
					\\
					&\quad\le (4c_4\zeta)^pT^{-p/2}\left( p^{p/2}+ p^p T^{-p/2} \right), \stepcounter{equation}\tag{\theequation}\label{eq: long comp 3}
				\end{align*}
				where $c_4\coloneqq \left(c_1a^{-2}\lor c_2a^{-1}\right)$. 
				Similarly, one obtains
				\begin{equation}\label{eq: long comp 4}
					\E_b\left[ \int_{y_i^*-\eta\Delta^\beta}^{y_i^*+\eta\Delta^\beta} \bigg\vert\frac{\hat{F}_T(z)}{\bar{\rho}_T(z)}-\frac{F(z)}{\rho(z)}\bigg\vert^p\d z \right]\leq 4c_4^p\eta\Delta^\beta T^{-p/2}\left( p^{p/2}+ p^p T^{-p/2} \right), 
				\end{equation}
				and putting everything together by combining the upper bounds \eqref{eq: long comp 00} to \eqref{eq: long comp 4} completes the proof.
			\end{proof}
			
			\begin{proof}[Proof of Proposition \ref{prop: existence hyp}]
				Fix $T>0$.
				Throughout the proof, we denote $\ep\coloneqq T^{-1/2}$. 
				\paragraph{Construction of the hypotheses}
				Letting $A_0\coloneqq y^*+M^\beta$, we define drift functions $b,\bb\colon \R\to\R$ by setting
				\begin{align*}
					\begin{split}
						b(x)\coloneqq \begin{cases}
							-x, & x\leq 0\\
							0, & 0<x\leq A_0\\
							-(x-A_0), &A_0<x
						\end{cases}
						\quad \text{ and }\quad  
						\bb(x)\coloneqq\begin{cases}
							-(1+\ep)^{-1}x, & x\leq 0\\
							0, & 0<x\leq A_0\\
							-(x-A_0), &A_0<x
						\end{cases},
					\end{split}
				\end{align*}
				respectively.
				Note that this definition ensures that, for large enough $T,$ $b,\bb\in\Sigma(\mathbf{C},A,\gamma)$. 
				Furthermore, we obtain for $0<x\leq A_0$
				\begin{align}\label{eq: xib}
					\xi_b(x)&=2\int_0^x \exp\left(-2\int_0^y b(u)\d u\right)\int_{-\infty}^y\exp\left(2\int_0^zb(u)\d u\right)\d z\d y\notag
					\\
					&=2\int_0^x \int_{-\infty}^0\exp\left(2\int_0^zb(u)\d u\right)\d z+ \int_{0}^y\exp\left(2\int_0^zb(u)\d u\right)\d z\d y\notag
					\\
					&=2\int_0^xy+ \frac{\sqrt{\pi}}{2}\d y\notag
					\\
					&=x^2+\sqrt{\pi}x
				\end{align}
				and
				\begin{align*}
					\xi_{\bb}(x)
					&=2\int_0^xy+ \int_{-\infty}^0\exp\left(2\int_0^z\bb(u)\d u\right)\d z\d y
					\\
					&=x^2+(1+\ep)\sqrt{\pi}x.
				\end{align*}
				Hence, defining the functions $f,g\colon(0,\infty)\to \R$ by setting
				\[	f(y)\coloneqq M-\vert y-y^*\vert^{1/\beta}
				\quad\mathrm{ and }\quad
				g(y)\coloneqq f(y)\xi_{\bb}(y), \]
				respectively, immediately yields $g\in \cG_{\bb}(y_1,\zeta,MM_2,\Delta_0,n,\eta,\beta)$.
				Finally, letting $\bar{g}\colonequiv f\xi_b$, analogously gives $\bar{g}\in \cG_b(y_1,\zeta,MM_2,\Delta_0,n,\eta,\beta)$.
				
				\paragraph{Verification of $\operatorname{(a)}$}
				It follows from Girsanov's theorem that, for large enough $T$,
				\begin{align*}
					\KL(\Pro_{b}^T,\Pro_{\bb}^T)
					&=\E_b\left[\log\left(\frac{\rho_b}{\rho_{\bb}}(X_0)\right)\right]+\frac 1 2 \E_b\left[\int_0^T\left(b(X_s)-\bb(X_s)\right)^2\d s\right]
					\\
					&\le c+\E_b\left[2\int_0^{X_0}\left(b(y)-\bb(y)\right)\d y\right]+\frac T 2 \E_b\left[\left(b(X_0)-\bb(X_0)\right)^2\right]
					\\
					&\leq c+\ep^2\frac T 2 \E_b\left[X_0^2\1_{(-\infty,0]}(X_0)\right]
					\\
					&\leq c,
				\end{align*} 
				where $c>0$ is a strictly positive constant, whose value changes from line to line.
				
				\paragraph{Verification of $\operatorname{(b)}$}
				Note that
				\begin{equation*}
					\frac{g(y)}{\xi_b(y)}=f(y)\left(1+\frac{\ep\sqrt{\pi}}{y+\sqrt{\pi}}\right)\eqqcolon h(y),
				\end{equation*}
				implying
				\begin{equation*}
					h'(y)=f'(y)\left(1+\frac{\ep\sqrt{\pi}}{y+\sqrt{\pi}}\right)-f(y)\frac{\ep\sqrt{\pi}}{(y+\sqrt{\pi})^2}.
				\end{equation*}
				Hence, we obtain for all $y\in\big[y^*-M^\beta,y^*\big)$ and small enough $\ep$,
				\begin{align*}
					h'(y)&< 2f'(y)-f(y)\frac{\ep\sqrt{\pi}}{(y^*+\sqrt{\pi})^2}
					\\
					&=\frac{2}{\beta}\vert y^*-y\vert^{1/\beta-1}-\left(M-\vert y^*-y\vert^{1/\beta}\right)\frac{\ep\sqrt{\pi}}{(y^*+\sqrt{\pi})^2}
					\\
					&<\vert y^*-y\vert^{(1-\beta)/\beta}\left(\frac{2}{\beta}+\frac{y^\ast\sqrt{\pi}}{(y^*+\sqrt{\pi})^2} \right)-\frac{\ep M\sqrt{\pi}}{(y^*+\sqrt{\pi})^2}
					\\
					&=\frac{M\sqrt{\pi}}{(y^*+\sqrt{\pi})^2}\left(\vert y^*-y\vert^{(1-\beta)/\beta}\frac{2(y^*+\sqrt{\pi})^2+\beta y^*\sqrt{\pi}}{M\sqrt{\pi}\beta} -\ep\right)
				\end{align*}
				Denoting
				\[
				c_5\coloneqq\left(\frac{M\beta \sqrt{\pi}}{4(y^*+\sqrt{\pi})^2+2\beta y^*\sqrt{\pi} }\right)^{\beta/(1-\beta)}, 
				\]
				this gives for $y^*-c_5\ep^{\beta/(1-\beta)}\leq y\leq y^*$ that
				\[
				h'(y)<-M\frac{\sqrt{\pi}}{2(y^*+\sqrt{\pi})^2}\ep,
				\]
				and, since $f'(y)< 0$ and $f(y)\geq 0$ for all $y^*<y\leq y^*+M^\beta$, it also follows in this case that $h'(y)<0$.
				Hence, if 
				\[
				\Phi_{\bb}(g)-\frac{g(y)}{\xi_{\bb}(y)}=\vert y-y^*\vert^{1/\beta}\leq \frac{c_5^{1/\beta} \ep^{1/(1-\beta)}}{2^{1/\beta}}, 
				\]
				we obtain
				\[
				\vert y-y^*\vert\leq \frac{c_5 \ep^{\beta/(1-\beta)}}{2},  
				\]
				implying	
				\begin{align*}
					\Phi_b(g)-h(y)&\geq h\left(y^*-c_5\ep^{\beta/(1-\beta)}\right)
					-h\left(y^*-c_5\ep^{\beta/(1-\beta)}/2\right)\\
					&=-\int_{y^*-c_5\ep^{\beta/(1-\beta)}}^{y^*-c_5\ep^{\beta/(1-\beta)}/2}h'(y) \d y
					\\&>c_5M\frac{\sqrt{\pi}}{4(y^*+\sqrt{\pi})^2}\ep^{1/(1-\beta)}.
				\end{align*}
				Thus,
				\[
				\bigg\{ \Phi_{\bb}(g)-\frac{g(y)}{\xi_{\bb}(y)}\leq \frac{c_5^{1/\beta} \ep^{1/(1-\beta)}}{2^{1/\beta}}\bigg\}\subseteq\bigg\{\Phi_b(g)-h(y)>c_5M\frac{\sqrt{\pi}}{4(y^*+\sqrt{\pi})^2}\ep^{1/(1-\beta)} \bigg\} 
				\]
				holds true.
				\paragraph{Verification of $\operatorname{(c)}$}
				For $y^*-M^\beta\leq y< y^*$,
				\begin{align*}
					h'(y)&\geq f'(y)-\frac{\ep}{\sqrt{\pi}}f(y)\\
					&>\frac{1}{\beta}\vert y^*-y\vert^{(1-\beta)/\beta}-\frac{M\ep}{\sqrt{\pi}},
				\end{align*}
				implying that $h'(y)>0$ for 
				$y\in\left[y_1,y^*-(\beta M\ep/\sqrt{\pi})^{\beta/(1-\beta)}\right]$.
				Thus, it holds 
				\[h(y^*_b)=\Phi_b(g)\quad 
				\implies\quad 
				y^*_b\in\left[y^*-(\beta M\ep/\sqrt{\pi})^{\beta/(1-\beta)},y^*-c_5\ep^{\beta/(1-\beta)}\right].
				\]Recall that $\bar{g}\equiv f\xi_b,$ which implies for any $c>1$
				\[
				\left\{\Phi_b(\bar{g})-\frac{\bar{g}}{\xi_b}\leq c\ep^{1/(1-\beta)} \right\}= 
				\left[y^*-c^\beta\ep^{\beta/(1-\beta)},y^*+c^\beta\ep^{\beta/(1-\beta)}\right]. 
				\]
				Hence, $y\in \left\{\Phi_b(\bar{g})-\bar{g}/\xi_b\leq c\ep^{1/(1-\beta)}\right\}$ implies for small enough $\ep$ by applying a Taylor expansion
				\begin{align*}
					\Phi_b(g)-h(y)&\leq  \left(c^\beta+(\beta M/\sqrt{\pi})^{\beta/(1-\beta)}\right)\ep^{\beta/(1-\beta)}\sup_{y\in[y^*-(\beta M\ep/\sqrt{\pi})^{\beta/(1-\beta)},y^*+c^\beta\ep^{\beta/(1-\beta)}] } (-h'(y))
					\\
					&\le\left(c^\beta+(\beta M/\sqrt{\pi})^{\beta/(1-\beta)}\right)\ep^{\beta/(1-\beta)}
					\\&\hspace*{3cm}\times\sup_{y\in[y^*-(\beta M\ep/\sqrt{\pi})^{\beta/(1-\beta)},y^*+c^\beta\ep^{\beta/(1-\beta)}] }  \left(\frac{2}{\beta}\vert y^*-y\vert^{(1-\beta)/\beta}+\ep\frac{M}{\sqrt{\pi}}\right)
					\\
					&\le\left(c^\beta+(\beta M/\sqrt{\pi})^{\beta/(1-\beta)}\right) \left((2 M\pi^{-1/2}\lor2c^{1-\beta}\beta^{-1})+\frac{M}{\sqrt{\pi}}\right)\ep^{1/(1-\beta)}
					\\
					&\le c\left(1+(\beta M/\sqrt{\pi})^{\beta/(1-\beta)}\right) \left((2 M\pi^{-1/2}\lor2\beta^{-1})+\frac{M}{\sqrt{\pi}}\right)\ep^{1/(1-\beta)}.
				\end{align*}
				Thus, denoting
				\[c_6\coloneqq \left(1+(\beta M/\sqrt{\pi})^{\beta/(1-\beta)}\right) \left(2( M\pi^{-1/2}\lor\beta^{-1})+\frac{M}{\sqrt{\pi}}\right),\]
				it holds for any $c>1$
				\[\left\{\Phi_b(\bar{g})-\frac{\bar{g}}{\xi_b}> c\ep^{1/(1-\beta)} \right\}\supseteq\left\{\Phi_b(g)-\frac{g}{\xi_b}>c_6c\ep^{1/(1-\beta)} \right\},  \]
				which implies $g\in\cG_{\bar{g}}(\kappa,T)$. 
				This concludes the proof.
			\end{proof}

Before proving the lower bound on cumulative regret given in Theorem 4.2, we state and prove the following crucial auxiliary result.

\begin{lemma}\label{lemma: KL cum}
For $A,\beta>0$, $\mathbf{C}\ge 1$, let $b,\bb\in\Sigma(\mathbf{C},A,\gamma)$ be given such that $\Vert b-\bb\Vert_\infty\leq c<\infty$ for some constant $c>0$. 
Then, for any general impulse control strategy $K$,
\begin{align*}
\KL(\P^T_{\bb,K},\P^T_{b,K})=\E_{\bb}\left[\frac 1 2 \int_0^T (\bb-b)^2(X^K_s)\d s\right],
\end{align*}
where $\P^T_{b,K}$ denotes the measure induced by $\X^K$ on the space of càdlàg functions $D[0,\infty)$, restricted to $\cF^W_T$, with $(\cF^W_t)_{t\geq0}$ denoting the natural filtration of the driving Wiener process in \eqref{eq: sde}, and $\P^T_{\bb,K}$ is defined analogously.
\end{lemma}
\begin{proof}
For $\bb,b$ fulfilling $\Vert \bb-b\Vert_\infty \leq c<\infty$, let 
\[Y_T\coloneqq\int_0^T \left(\bb-b\right)(X_s^K)\d W_s. \]
Then, under the measure $\Q$ on $\mathcal{F}^W_T$ satisfying
\[\frac{\d \Q}{\d \P_{b,K}^T}=\exp\left( \int_0^T \left(\bb-b\right)(X_s^K)\d W_s- \frac 1 2 \int_0^T \left(\bb-b\right)^2(X_s^K)\d s \right), \]
Girsanov's theorem implies that
\[\tilde{W}_t\coloneq W_t-\int_0^t \left(\bb-b\right)(X_s^K)\d s,\quad 0\leq t\leq T, \]
is a Wiener process.
This is true since the Novikov condition is trivially fulfilled and because $\X^K$ is adapted to $(\cF^W_t)_{t\geq0}$. 
Thus, under $\Q$, it holds for $\tau_n\leq t< \tau_{n+1}\leq T$ that
\begin{align*}
X^K_t&=\int_{\tau_n}^t b(X^K_s)\d s+\int_{\tau_n}^t\d W_s\\
&=\int_{\tau_n}^t \bb(X^K_s)\d s+\int_{\tau_n}^t \d \tilde{W}_s,
\end{align*}
and hence $\Q=\P^T_{\bb,K}$. 
Thus, we get for the Kullback--Leibler divergence
\begin{align*}
\KL(\P^T_{\bb,K},\P^T_{b,K})&=\E_{\bb}\left[\log\left(\frac{\d \P^T_{\bb,K}}{\d \P^T_{b,K}} \right) \right]
		\\
&=\E_{\bb}\left[\int_0^T (\bb-b)(X^K_s)\d W_s-\frac 1 2 \int_0^T (\bb-b)^2(X^K_s)\d s\right]
\\
&=\E_{\bb}\left[\int_0^T (\bb-b)(X^K_s)\d \tilde{W}_s+\frac 1 2 \int_0^T (\bb-b)^2(X^K_s)\d s\right]
\\
&=\E_{\bb}\left[\frac 1 2 \int_0^T (\bb-b)^2(X^K_s)\d s\right].
\end{align*}
\end{proof}	
			\begin{proof}[Proof of Theorem \ref{thm: lower bound cum}.]
Fix an impulse control strategy $K$ wrt $\cG(y_1,\zeta,MM_2,\mathbf{C},A,\gamma)$, 
and define $b,\bar{b}$ as in the proof of Theorem \ref{thm: lower bound gen}. 
Then, it holds $b,\bb\in\Sigma(\mathbf{C},A,\gamma)$ for large enough $T$, and Lemma \ref{lemma: KL cum} implies that
\begin{align}\label{eq: kl cum bound}
		\KL(\P^T_{\bb,K},\P^T_{b,K})\leq M^2,
\end{align}
 where $\P^T_{\bb,K},\P^T_{b,K}$ are defined in Lemma \ref{lemma: KL cum}.
Furthermore, choosing $g$ again as in the proof of Theorem \ref{thm: lower bound gen} also implies $g\in \cG(y_1,\zeta,MM_2,\mathbf{C},A,\gamma)$ for large enough $T$, and 
 \[\inf_{y\in[y_1,a]}\left(\Phi_b(g)-\frac{g}{\xi_b}(y)\right)\geq cT^{-1/2} ,\quad \inf_{y\in[a,\zeta]}\left(\Phi_{\bb}(g)-\frac{g}{\xi_{\bb}}(y)\right)\geq cT^{-1/2}, \]
 for some constant $c>0$ specified in the proof of Theorem \ref{thm: lower bound gen}.
Defining
	\begin{align*}
		\mathcal{T}_{1}&\coloneqq\{n\in \N: \tau_{n-1} \leq T, y_n\leq a \},\quad T_{1}\coloneq \sum_{n\in\mathcal{T}_1} (\tau_n\land T-\tau_{n-1}),
		\\
		\mathcal{T}_{2}&\coloneqq\{n\in \N: \tau_{n-1} \leq T, y_n>a \},\quad
		T_{2}\coloneq \sum_{n\in\mathcal{T}_1} (\tau_n\land T-\tau_{n-1}),
	\end{align*} 
it then holds
	\begin{align}
		\E_b\left[\sum_{n:\tau_n\leq T}g(X_{\tau_n-}^K)\right]&\leq \E_b\left[\sum_{n\in \mathcal{T}_{1}}g(X_{\tau_n-}^K)\right]+\E_b\left[\sum_{n\in \mathcal{T}_{2}}g(X_{\tau_n-}^K)\right]\nonumber
		\\\nonumber
		&=  \E_b\left[\sum_{n\in \mathcal{T}_{1}}\frac{g(X_{\tau_n-}^K)}{\E_b[\tau_{n}-\tau_{n-1}\vert \mathcal{F}_{\tau_{n-1}}]}\E_b[\tau_{n}-\tau_{n-1}\vert \mathcal{F}_{\tau_{n-1}}]\right]\\&\quad+\E_b\left[\sum_{n\in \mathcal{T}_{1}}\frac{g(X_{\tau_n-}^K)}{\E_b[\tau_{n}-\tau_{n-1}\vert \mathcal{F}_{\tau_{n-1}}]}\E_b[\tau_{n}-\tau_{n-1}\vert \mathcal{F}_{\tau_{n-1}}]\right]\nonumber
		\\
		&=  \E_b\left[\sum_{n\in \mathcal{T}_{1}}\frac{g(y_n)}{\xi_b(y_n)}\E_b[\tau_{n}-\tau_{n-1}\vert \mathcal{F}_{\tau_{n-1}}]\right]+\E_b\left[\sum_{n\in \mathcal{T}_{1}}\frac{g(y_n)}{\xi_b(y_n)}\E_b[\tau_{n}-\tau_{n-1}\vert \mathcal{F}_{\tau_{n-1}}]\right]\nonumber
		\\
		&\leq\Phi_b(g)  \E_b\left[\sum_{n\in \mathcal{T}_{1}}\E_b[\tau_{n}-\tau_{n-1}\vert \mathcal{F}_{\tau_{n-1}}]\right]+\E_b\left[\sum_{n\in \mathcal{T}_{1}}\frac{g(y_n)}{\xi_b(y_n)}\E_b[\tau_{n}-\tau_{n-1}\vert \mathcal{F}_{\tau_{n-1}}]\right]\nonumber
		\\\nonumber
		&=\Phi_b(g)  \E_b\left[\sum_{n:\tau_{n-1}\leq T}\E_b[\tau_{n}-\tau_{n-1}\vert \mathcal{F}_{\tau_{n-1}}]\right]\\\nonumber&\quad+\E_b\left[\sum_{n\in \mathcal{T}_{1}}\left(\frac{g(y_n)}{\xi_b(y_n)}-\Phi_b(g)\right)\E_b[\tau_{n}-\tau_{n-1}\vert \mathcal{F}_{\tau_{n-1}}]\right]
		\\\nonumber
		&\leq \Phi_b(g)  \E_b\left[\sum_{n:\tau_{n-1}\leq T}\E_b[\tau_{n}-\tau_{n-1}\vert \mathcal{F}_{\tau_{n-1}}]\right]-cT^{-1/2}\E_b\left[\sum_{n\in \mathcal{T}_{1}}\E_b[\tau_{n}-\tau_{n-1}\vert \mathcal{F}_{\tau_{n-1}}]\right]
		\\
		&\leq  \Phi_b(g)(T+\xi_b(\zeta))-cT^{-1/2}\E_b\left[T_1 \right].\label{eq: cum 1}
	\end{align}
Analogous arguments give	
\begin{equation}\label{eq: cum 2}
\E_{\bb}\left[\sum_{n:\tau_n\leq T}g(X_{\tau_n-}^K)\right]
\le \Phi_{\bb}(g)(T+\xi_{\bb}(\zeta))-cT^{-1/2}\E_{\bb}\left[T_2 \right].
\end{equation}
Now, applying the Bretagnolle--Huber inequality \eqref{bh ineq} together with \eqref{eq: kl cum bound} implies
\begin{equation*}
\P_{b}(T_1\geq T/2)+\P_{\bb}(T_1< T/2)\ge\frac 1 2 \exp(-M^2),
\end{equation*}
which, combined with \eqref{eq: cum 1} and \eqref{eq: cum 2}, gives for large enough $T$
\begin{align*}
&\Phi_b(g)T- \E_b\left[\sum_{n:\tau_n\leq T}g(X_{\tau_n-}^K)\right]+\Phi_{\bb}(g)T- \E_{\bb}\left[\sum_{n:\tau_n\leq T}g(X_{\tau_n-}^K)\right]\\
&\quad\ge cT^{-1/2}(\E_b\left[T_1\right]+\E_{\bb}\left[T_2\right])-\Phi_b(g)\xi_b(\zeta)-\Phi_{\bb}(g)\xi_{\bb}(\zeta)\\
&\quad\ge c\sqrt{T} (\P_b(T_1\geq T/2)+\P_{\bb}(T_1< T/2))-\Phi_b(g)\xi_b(\zeta)-\Phi_{\bb}(g)\xi_{\bb}(\zeta)\\
&\quad\ge c\sqrt{T}-\Phi_b(g)\xi_b(\zeta)-\Phi_{\bb}(g)\xi_{\bb}(\zeta)\\
&\quad\ge c\sqrt{T},
\end{align*}
where the value of $c>0$ changes from line to line. 
As $K$ was arbitrary and $2\max(a,b)\geq a+b$ holds for any $a,b\in\R$, this concludes the proof.
\end{proof}

\paragraph{Acknowledgement}
ND and CS are grateful for financial support of Sapere Aude: DFF-Starting Grant 0165-00061B ``Learning diffusion dynamics and strategies for optimal control''.

\printbibliography
			
\end{document}